\documentclass{amsart}

\usepackage{amsmath,amssymb}
\usepackage{amsfonts}
\usepackage{graphicx}
\usepackage{latexsym}
\usepackage{array}
\usepackage{subfig}
\usepackage{stmaryrd}
\usepackage{color}
\usepackage{multicol}
\usepackage{vwcol}
\usepackage{hyperref} 

\newcommand{\balpha}{{\boldsymbol{\alpha}}}
\newcommand{\bbeta}{{\boldsymbol{\beta}}}

\newcommand{\blam}{{\boldsymbol{\lambda}}}

\newtheorem{theorem}{Theorem}[section]
\newtheorem{lemma}[theorem]{Lemma}
\newtheorem{corollary}{Corollary}[section]      

\theoremstyle{definition}
\newtheorem{definition}[theorem]{Definition}

\theoremstyle{remark}
\newtheorem{remark}[theorem]{Remark}
\newtheorem{assumption}[theorem]{Assumption}
\numberwithin{equation}{section}

\newcommand{\mc}[1]{{\mathcal{#1}}}

\begin{document}

\title[$H^m$-Nonconforming FEM Spaces for arbitrary $\mathbb{R}^n$]{A
construction of canonical nonconforming finite element spaces for
elliptic equations of any order in any dimension}
\thanks{This study is supported in part by the National Natural
Science Foundation of China grant No. 12222101 and the Beijing Natural
Science Foundation No. 1232007.}



\author[J.~Li]{Jia Li}
\address[J.~Li]{School of Mathematical Sciences, Peking University,
  Beijing, 100871, P.R. China}
\email{lijia2002@stu.pku.edu.cn}

\author[S.~Wu]{Shuonan~Wu}
\address[S.~Wu]{School of Mathematical Sciences, Peking University,
  Beijing, 100871, P.R. China}
\email{snwu@math.pku.edu.cn}


\date{}


\begin{abstract}
A unified construction of canonical $H^m$-nonconforming finite elements is developed for $n$-dimensional simplices for any $m, n \geq 1$. Consistency with the Morley-Wang-Xu elements [Math. Comp. 82 (2013), pp. 25-43] is maintained when $m \leq n$. In the general case, the degrees of freedom and the shape function space exhibit well-matched multi-layer structures that ensure their alignment. Building on the concept of the nonconforming bubble function, the unisolvence is established using an equivalent integral-type representation of the shape function space and by applying induction on $m$. The corresponding nonconforming finite element method applies to $2m$-th order elliptic problems, with numerical results for $m=3$ and $m=4$ in 2D supporting the theoretical analysis.
\end{abstract}

\maketitle

\section{Introduction} \label{sec:intro}
Numerical methods for solving $2m$-th order elliptic equations in $n$-dimensional domains have long been a fundamental problem in finite element methods.

A straightforward approach is to construct a finite element space with piecewise polynomials that satisfies $H^m$ conformity, namely, $C^{m-1}$ continuity. 
However, exploring this canonical setting for arbitrary $m$ or $n$ is far from trivial. 
For simplicial meshes when $n = 2$, researchers have progressively developed $H^2$-conforming element \cite{ciarlet1978finite} (known as Argyris element), and $H^m$-conforming elements for arbitrary $m$ \cite{vzenivsek1970interpolation,bramble1970triangular}.
However, as the dimensionality increases, the geometric complexity of simplex elements escalates rapidly, making the problem extremely challenging. 
Progress has been made only in specific cases over the years, such as for $m=2, n=3$ \cite{zhang2009family, walkington2014c}, for $m=2, n=4$ \cite{zhang2016family}, and arbitrary $m$ for $n=3$ or $n=4$ \cite{zhang2022nodal}. 
It was not until recently that Hu, Lin, and Wu \cite{hu2023construction} provided a construction of conforming elements for $m,n\geq1$, addressing this open problem in the canonical setting. Notably, in their construction, the polynomial degree starts from $(m-1)2^n+1$. 
Subsequent work \cite{hu2024condition} confirmed that this is indeed the minimal possible polynomial degree. 
A geometric interpretation of this fundamental result is provided in \cite{chen2021geometric}. An alternative approach is constructing conforming finite elements on rectangular meshes for arbitrary $m$ and $n$ (see \cite{hu2015minimal}).

Within this canonical setting of piecewise polynomials, another research direction is the nonconforming finite element method. 
Although the corresponding finite element space does not form a subspace of $H^m$, it is designed with a weak continuity that, when combined with a simple broken bilinear form, proves effective. 
One direct advantage over conforming finite element method is the ability to use {\it lower-degree} polynomial spaces locally, which offers computational benefits.
For instance, the classical two-dimensional Morley element uses only piecewise quadratic polynomials to solve fourth-order ($m=2$) elliptic problems, making it minimal in terms of polynomial approximation.
For $n=2$, \cite{hu2016canonical} introduced a family of nonconforming
finite elements for $m \geq 3$ with polynomial degrees of $\max\{2m-3,
4\}$, exhibiting first-order convergence in the broken $H^m$ norm.
In the case of $n=3$, \cite{hu2020family} constructed nonconforming elements of convergence order $k-1$ for $H^2$ problems by enriching $\mathcal{P}_k$ with face bubble functions from $\mathcal{P}_{k+4}$, $k\geq 3$.

Extending nonconforming elements to arbitrary dimensions $n$ is also a significant challenge. In this direction, there are the Morley-type elements \cite{wang2006morley} and the Zienkiewicz-type elements \cite{wang2007new} for $m=2$ on simplicial meshes. On $n$-rectangular meshes, \cite{wang2007some} introduced elements for $m=2$, and \cite{jin2023two} for $m=3$. 

The first comprehensive study of nonconforming elements that simultaneously addresses general \( m \) and \( n \) is the well-known Morley-Wang-Xu (MWX) family of elements \cite{wang2013minimal}, which elegantly integrates simplicial geometry, degree of freedom allocation, and convergence analysis.
The MWX elements are minimal in the sense that the local polynomial is $\mathcal{P}_m$, but they have a fundamental limitation: $m \leq n$. 
In \cite{wu2019nonconforming}, recognizing the similarity in degrees of freedom on sub-simplices, one of the authors, Wu, along with Xu, proposed a construction for the case $m = n+1$ by enriching $\mathcal{P}_{n+1}$ with $n$ volume bubble functions. To some extent, the unisolvence of the element mirrors the process of taking traces in Sobolev spaces, essentially forming an induction argument on $n$ along the line $m = n+1$.
However, the general construction of nonconforming elements for arbitrary $m$ and $n$ remains an open problem. In this paper, we provide a solution to this open problem, employing a different approach from previous works.

The key advantage of the new family of nonconforming finite elements is that, in the proof of unisolvence, an induction argument can be naturally carried out with respect to $m$ for any given $n$.
Specifically, the unisolvence of the element mirrors the process of
taking derivatives from $H^m$ to $H^{m-1}$ in Sobolev spaces. Achieving this at the discrete level requires careful coordination between the construction of both the degrees of freedom and the shape function space, which represents two key technical contributions of this paper:

\begin{enumerate}
\item  In constructing the degrees of freedom, we carefully distribute them across different subsimplices, ensuring that any function $v$ that vanishes at the degrees of freedom for $m$ will also have its first derivative vanish at the degrees of freedom for $m-1$ (see Lemma \ref{lm:LW-derivative}). This property, which is purely associated with the degrees of freedom, also holds for the MWX elements. Our design naturally extends the MWX-type degrees of freedom into a {\it multi-layer structure}, where the number of layers depends on the ceiling of $\frac{m}{n}$.
\item Based on this multi-layer structure of degrees of freedom, we first design a corresponding shape function space in \eqref{eq:shape} with a similar structure and provide a general equivalent integral-type representation. 
Notably, the differentiation required in the induction argument with respect to $m$ naturally fits into this integral-type representation.
Another crucial aspect of unisolvence is the introduction of the concept of {\it nonconforming bubble function}, which leads to a special form of equivalent integral-type representation (see Theorem \ref{tm:bl}). This plays a key role in the induction argument when the layers of degrees of freedom corresponding to $m$ and $m-1$ differ.
\end{enumerate}

By utilizing the key techniques outlined above, the proof of unisolvence is both universal and elegant (see Theorem \ref{tm:unisolvence}). Moreover, the elements are consistent with the MWX elements when $m \leq n$, and they degenerate to conforming elements when $n=1$. In terms of the number of local degrees of freedom, it is notably low. For the two-dimensional case, due to \eqref{eq:number-dof}, when $m=3,4,5$ the degrees of freedom are $\dim \mathcal{P}_3 + \dim \mathcal{P}_1 - \dim\mathcal{P}_0 = 12$, $\dim \mathcal{P}_4 + \dim \mathcal{P}_2 - \dim\mathcal{P}_1 = 18$, and $\dim \mathcal{P}_5 + \dim \mathcal{P}_3 - \dim\mathcal{P}_2 + \dim \mathcal{P}_1 - \dim\mathcal{P}_0 = 27$, respectively, while for the Hu-Zhang element \cite{hu2016canonical}, they are $15$, $21$, and $36$, respectively.

As a direct application of this new family of elements, we apply it to solve the following $m$th-Laplace equations with homogeneous boundary conditions:
\begin{equation} \label{eq:m-harmonic}
\left\{
\begin{aligned}
(-\Delta)^m u &= f \qquad \mbox{in }\Omega, \\
\frac{\partial^k u}{\partial \nu^k} &= 0 \qquad \mbox{on }\partial
\Omega, \quad 0 \leq k \leq m-1.
\end{aligned}
\right.
\end{equation}

Lastly, we note that in addition to the design of conforming or nonconforming elements, various numerical methods are available for solving high-order elliptic equations. These include $C^0$-interior penalty discontinuous Galerkin methods \cite{brenner2005c0, gudi2011interior, chen2022c0}, stabilized nonconforming methods \cite{wu2017pm}, and mixed methods \cite{arnold1985mixed, li1999full, li2006optimal, droniou2017mixed, schedensack2016new}. Beyond techniques that operate on discrete variational forms, another approach involves extending the local function space beyond polynomials, such as the rapidly evolving and increasingly important virtual element methods (VEM) \cite{beirao2013basic, beirao2014hitchhiker}. For conforming VEM, we refer to \cite{antonietti2020conforming, da2020c1, chen2022conforming}, and for nonconforming VEM, see \cite{huang2020nonconforming, chen2020nonconforming}.

The outline of the paper is as follows. Section \ref{sec:preliminary} begins with some preliminary concepts and provides an overview of the MWX elements and their design principles. In Section \ref{sec:element}, we define our elements and introduce a general integral-type representation of the shape function space. We establish the unisolvence for any $m,n \geq 1$ in Section \ref{sec:unisolvence}. Section \ref{sec:2m} applies the proposed elements to solve the typical $2m$-th order elliptic problems \eqref{eq:m-harmonic}. Section \ref{sec:numerical} provides some numerical results to validate our theoretical findings. Finally, Section \ref{sec:remarks} offers concluding remarks.


\section{Preliminary} \label{sec:preliminary}

This section introduces some basic notation and reviews the principles behind the construction of the Morley-Wang-Xu (MWX) elements \cite{wang2013minimal}, along with a discussion of key concepts in the convergence theory.

\subsection{Notation}
For an $n$-dimensional multi-index $\balpha = (\alpha_1, \cdots, \alpha_n)$, we define 
$$ 
|\balpha| = \sum_{i=1}^n \alpha_i, \quad 
\partial^\balpha = \frac{\partial^{|\balpha|}}{\partial x_1^{\alpha_1}
\cdots \partial x_n^{\alpha_n}}.
$$ 
Given an integer $k\geq 0$ and a bounded domain $G \subset \mathbb{R}^n$ with boundary $\partial G$, let $H^k(G), H_0^k(G),
\|\cdot\|_{k,G}$, and $|\cdot|_{k,G}$ denote the usual Sobolev spaces, norm, and semi-norm, respectively (c.f. \cite{adams2003sobolev}).

Assume that $\Omega$ is a bounded polyhedron domain of $\mathbb{R}^n$. Let $\mathcal{T}_h$ be a conforming and shape-regular simplicial triangulation of $\Omega$, and let $\mathcal{F}_h$ denote the set of all faces of $\mathcal{T}_h$.
Define $\mathcal{F}_h^i := \mathcal{F}_h \setminus \partial \Omega$.  Here, $h := \max_{T\in \mathcal{T}_h} h_T$, where $h_T$ is the diameter of $T$ (cf. \cite{ciarlet1978finite,
    brenner2007mathematical}).  We assume that $\mathcal{T}_h$ is
quasi-uniform, namely 
$$    
\exists \eta>0 \mbox{~~such that~~} \max_{T\in \mathcal{T}_h}
\frac{h}{h_T} \leq \eta,
$$
where the constant $\eta$ is independent of $h$. For the triangulation $\mathcal{T}_h$, and for $v \in L^2(\Omega)$ with $v|_T \in H^k(T)$ for all $T\in \mathcal{T}_h$, we define $\partial^{\balpha}_h v$ as the piecewise partial derivatives of $v$ for $|\balpha| \leq k$. We also define
$$ 
\|v\|_{k,h}^2 := \sum_{T \in \mathcal{T}_h}\|v\|_{k,T}^2, \quad
|v|_{k,h}^2 := \sum_{T\in \mathcal{T}_h} |v|_{k,T}^2.
$$ 

The linear space spanned by $v_1, \ldots, v_d$ is denoted by $\langle v_1, \ldots, v_d \rangle$. 
In the analysis, $C$ represents a generic positive constant that may vary between occurrences but is independent of the mesh size $h$. The notation $X \lesssim Y$ signifies that $X \leq C Y$.

\subsection{Review of Morley-Wang-Xu (MWX) elements}
To better understand the design principles of $H^m$-nonconforming elements, we first review the well-known MWX elements \cite{wang2013minimal}.

Let $T$ be an $n$-simplex. Following the description in \cite{ciarlet1978finite,brenner2007mathematical}, a finite element can be represented by a triple $(T, P_T, D_T)$, where $T$ denotes the geometric shape of the element, $P_T$ is the shape function space, and $D_T$ is the set of the degrees of freedom (DOF) that is $P_T$-unisolvent. For the family of MWX elements, the shape function spaces are minimal in the sense that $P_T^{(m,n)} = \mathcal{P}_m(T)$, but with the {\it inherent limitation} that $m \leq n$.

Let $\mathcal{F}_{T,k}$ denote the set of all $(n-k)$-dimensional sub-simplices of $T$. For any $F \in \mathcal{F}_{T,k}$, let $|F|$ represent its $(n-k)$-dimensional measure, and let $\nu_{F,1}, \ldots, \nu_{F,k}$ be linearly independent unit vectors orthogonal to the tangent space of $F$. Specifically, when $k = n$, $F$ represents a vertex and $|F| = 1$. 

For $k\geq 1$, let $A_k$ be the set consisting of all multi-indexes $\balpha$ with $\sum_{i=k+1}^n\alpha_i =0$. The core principle in constructing the MWX elements is the precise distribution of DOF across the sub-simplices. 
Specifically, the set of DOF of MWX element is defined as given in \cite[Equ. (2.3)--(2.4)]{wang2013minimal}: 
\begin{equation}  \label{eq:MWX-DOF}
    D_T^{(m,n)} := \{d_{T,F,\balpha} : \balpha \in A_k \text{ with } |\balpha| = m-k, 1\leq k\leq m, F\in \mc F_{T,k}\},
\end{equation}
where for any $1 \leq k \leq m$, any $(n-k)$-dimensional sub-simplex $F \in \mathcal{F}_{T,k}$, and $\balpha \in A_k$ with $|\balpha| = m-k$, 
\begin{equation} \label{eq:MWX-DOF2}
 d_{T,F,\balpha}(v) := \frac{1}{|F|}\int_F \frac{\partial^{m-k} v}{\partial \nu^{\alpha_1}_{F,1} \cdots \partial\nu^{\alpha_k}_{F,k}} 
 \quad \forall v\in H^{m}(T).
\end{equation}

In addition to being naturally defined on $H^m(T)$ (cf. \cite{adams2003sobolev}), set of DOF \eqref{eq:MWX-DOF} has the following notable properties.
\begin{lemma}[Lemma 2.1 in \cite{wang2013minimal}]  \label{lm:MWX-derivative}
    If all the degrees of freedom defined in \eqref{eq:MWX-DOF}
    vanish, then for any $0\leq k \leq m$, any $(n-k)$-dimensional
    sub-simplex $F\in \mc F_{T,k}$, it holds that for any $v \in H^m(T)$,
    \begin{equation} \label{eq:MWX-derivative}
      \frac{1}{|F|}\int_F \nabla^{m-k} v = 0,
    \end{equation}
    where $\nabla^j$ is the $j$th Hessian tensor for any integer $j \geq 0$.
\end{lemma}
Let us remark here that the original proof did not consider the case in which $k=0$ (i.e., $F = T$). However, this case can be handled by simply performing integration by parts over $T$. 

Clearly, \eqref{eq:MWX-derivative} implies that all the degrees of freedom in $D_T^{(m,n)}$ vanish, making these two conditions equivalent. Such an equivalence demonstrates that any function vanishing in $D_T^{(m,n)}$ will also vanish in $D_T^{(m-1,n)}$ after taking any first-order derivative. Combined with the minimal shape function space of the MWX elements, the unisolvence of MWX elements can thus be directly derived through induction on $m$ (see \cite[Lemma 2.2]{wang2013minimal}).

For a given nonconforming element and triangulation $\mathcal{T}_h$, the corresponding finite element space is denoted by $V_h^{(m,n)}$. For the $2m$-th order elliptic problem, the key concepts governing the convergence of the nonconforming finite element method are weak continuity and the weak zero-boundary condition, which are defined as follows.

\begin{definition}[weak continuity \& weak zero-boundary condition] \label{df:weak-continuity}
    For $|\balpha|<m$, $F$ be an $(n-1)$-dimension sub-simplex of
    $T\in \mathcal{T}_h$. Then, for any $v_h \in V_h^{(m,n)}$,
    $\partial_h^{\balpha}v_h$ is continuous at a point on $F$ at least.
    If $F\subset \partial \Omega$, then $\partial_h^{\balpha}v_h$ vanishes
    at point on $F$ at least.  
\end{definition}

For the MWX elements, these properties can be directly derived from Lemma \ref{lm:MWX-derivative}, further highlighting the critical importance of this lemma.
In the following sections, we will extend the design of degrees of freedom for the MWX elements
to a universal construction of nonconforming finite elements with arbitrary $m,n \geq 1$, such that
properties in Lemma \ref{lm:MWX-derivative} is preserved. 

\section{Nonconforming Finite Elements} \label{sec:element}
In this section, we shall construct a universal family of nonconforming finite elements for arbitrary $H^{m}$ in any dimension $n \geq 1$.

\subsection{Degrees of freedom} 

Given an $n$-simplex $T$ with vertices $a_i, 0 \leq i \leq n$, let
$\lambda_0, \lambda_1, \ldots, \lambda_{n}$ be the barycenter
coordinates of $T$. Below, we present the degrees of freedom for arbitrary $m, n \geq 1$, which can be viewed as a natural extension of the MWX elements given in \eqref{eq:MWX-DOF2}.

For $1\leq k\leq n$, any $(n-k)$-dimensional sub-simplex $F\in \mc F_{T,k}$ and $\balpha \in A_k$ with $|\balpha| = m-\ell n-k$ (here, $\ell \in \mathbb{N}$ such that $m- \ell n-k\geq 0$),  
we define
\begin{equation} \label{eq:DOF2}
  d_{T,F,\balpha}(v) := \frac{1}{|F|}\int_F \frac{\partial^{m-\ell n-k} v}{\partial \nu^{\alpha_1}_{F,1} \cdots \partial\nu^{\alpha_k}_{F,k}} \quad \forall v\in H^{m}(T).
\end{equation}
For $k = n$, i.e. $F=a$ , $\int_F f = f(a)$ by convention represents the vertex evaluation. 

By the Sobolev embedding theorem (cf. \cite{adams2003sobolev}), $d_{T,F,\balpha}$ are continuous linear functionals on $H^m(T)$. 
Then, the set of the degrees of freedom is
\begin{equation}  \label{eq:DOF}
  \begin{aligned}
 D_T^{(m,n)} := \{ & d_{T,F,\balpha} : \balpha \in A_k \text{ with } | \balpha | = m-\ell n-k \\
 & \text{ where }\ell \in\mathbb{N} \text{ s.t. } m-\ell n-k\geq 0,F\in \mc F_{T,k},1\leq k\leq n\}.
  \end{aligned}
\end{equation}
We also number the local degrees of freedom by
 $$ d_{T,1},d_{T,2}, \ldots, d_{T,J},$$
 where $J$ is the number of local degrees of freedom. The degrees of freedom are depicted in Table \ref{tab:DOF}.
 
We provide a layer-wise characterization of \eqref{eq:DOF}. To this end, we first define the ceiling function of a real number $x$, denoted by $\lceil x \rceil$, as the smallest integer greater than or equal to $x$. Then, the value of $\ell$ in \eqref{eq:DOF} can only range between $0$ and $\lceil \frac{m}{n} \rceil - 1$. Accordingly, we can categorize the degrees of freedom \eqref{eq:DOF2} into different layers based on the value of $\ell$, as follows:
\begin{equation} \label{eq:DOF-layer}
\begin{aligned}
D_T^{(m,n)} & = \bigcup_{\ell=0}^{\lceil \frac{m}{n} \rceil - 1}
\left\{ 
\begin{aligned}
 d_{T,F,\balpha} : ~& \balpha \in A_k \text{ with } | \balpha | = m-\ell n-k, \\
 & F\in \mc F_{T,k},1\leq k\leq \min\{n, m-\ell n\}
 \end{aligned}
 \right\} \\
 & :=  \bigcup_{\ell=0}^{\lceil \frac{m}{n} \rceil - 1} D_{T,\ell}^{(m,n)}.
 \end{aligned}
\end{equation}
 
 \begin{remark}[natural extension of MWX elements]
By comparing \eqref{eq:DOF2} with \eqref{eq:MWX-DOF2}, it is evident that in the case considered by the MWX elements, where $m \leq n$, we have $\lceil \frac{m}{n} \rceil - 1=0$. In other words, the MWX elements only involve degrees of freedom of the $0$-th layer, which is consistent with our definition given in \eqref{eq:DOF-layer} when $m \leq n$.
 \end{remark}
 
 \begin{table}[!htbp]
\caption{arbitrary $m,n \geq 1$: diagrams of the finite elements. The case in which $m \leq n$ coincides with the Morley-Wang-Xu element.}
\begin{center}
\begin{tabular}{>{\centering\arraybackslash}m{.6cm} |
>{\centering\arraybackslash}m{2.8cm} |
>{\centering\arraybackslash}m{2.8cm} |
>{\centering\arraybackslash}m{2.8cm} @{}m{0pt}@{} }
\hline
$m \backslash n$ & 1 & 2 & 3 \\ \hline
1 &
\includegraphics[width=1.1in]{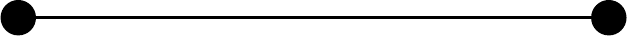}
& 
\includegraphics[width=1.1in]{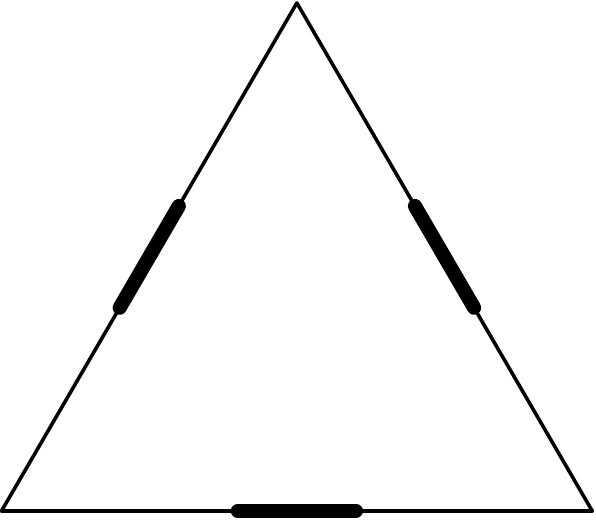}
& 
\includegraphics[width=1.1in]{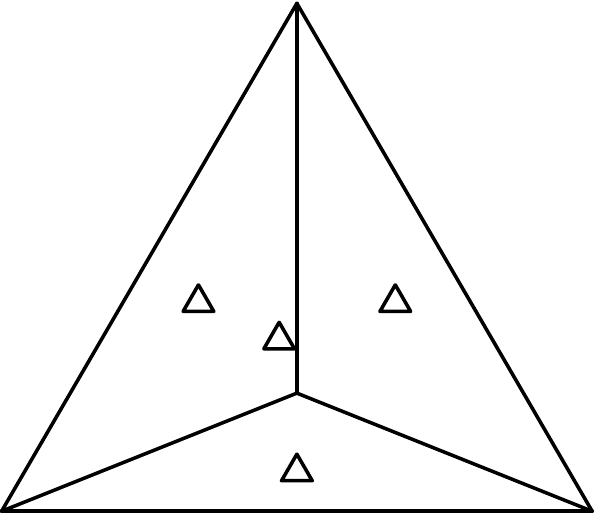} \\ 
\hline 
2 & 
\includegraphics[width=1.1in]{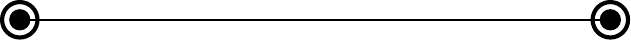}
& 
\includegraphics[width=1.1in]{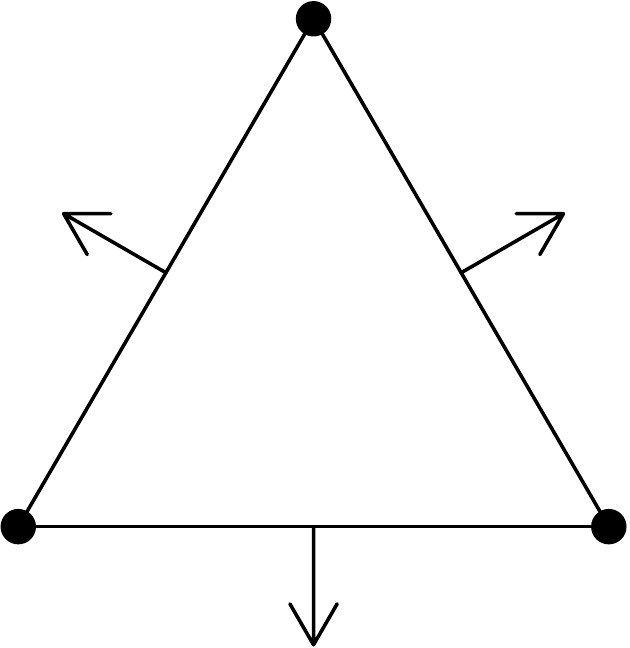}
&
\includegraphics[width=1.1in]{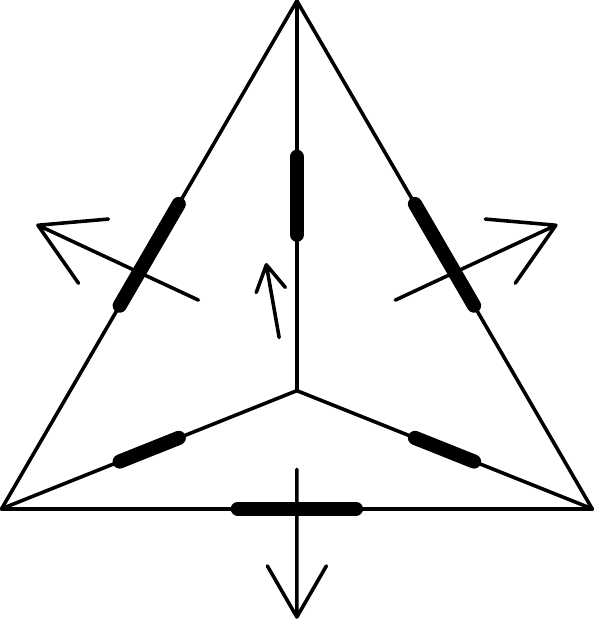}
\\ \hline 
3 &
\includegraphics[width=1.1in]{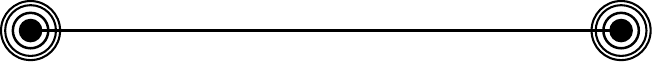} 
&
\includegraphics[width=1.1in]{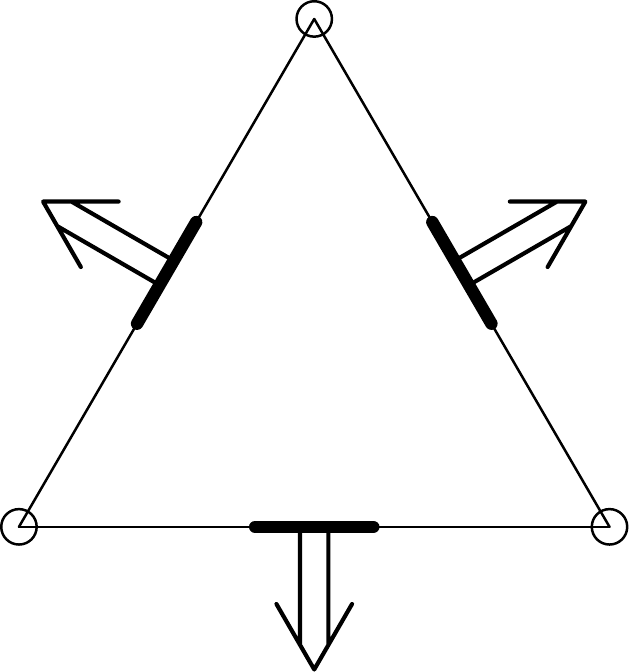}
&
\includegraphics[width=1.1in]{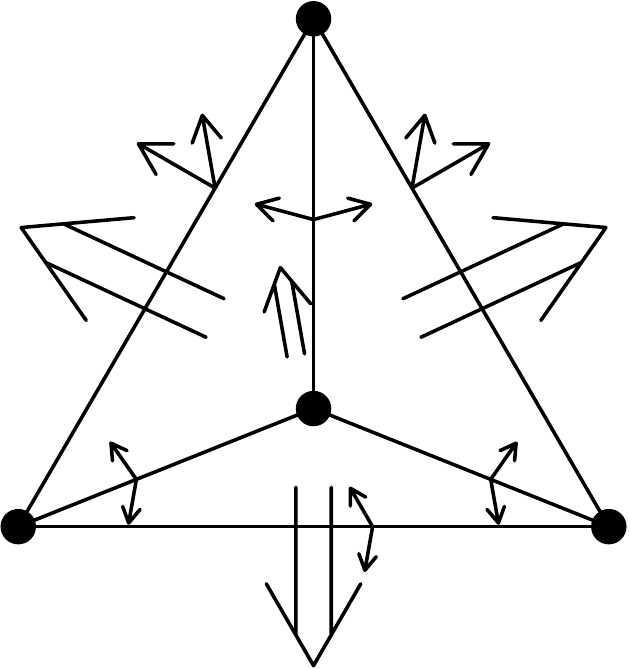}
\\ \hline
4 & 
\includegraphics[width=1.1in]{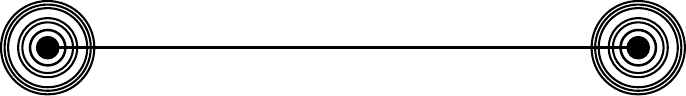}
&
\includegraphics[width=1.1in]{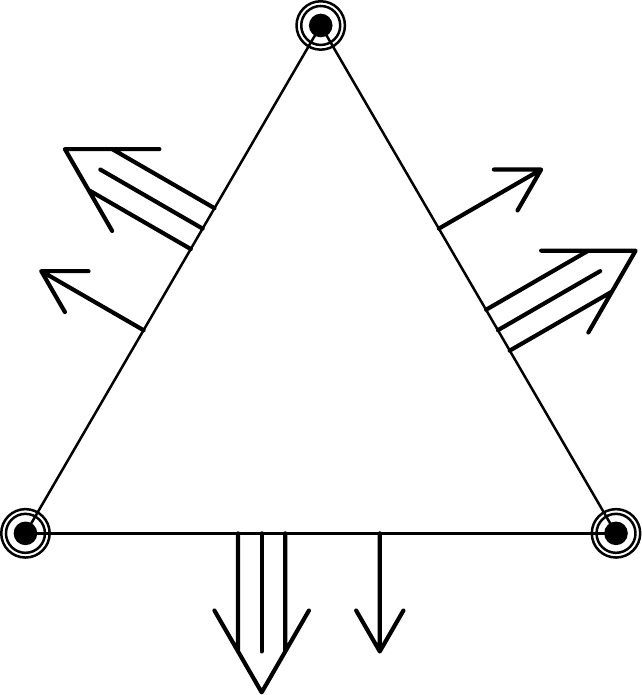}
&
\includegraphics[width=1.1in]{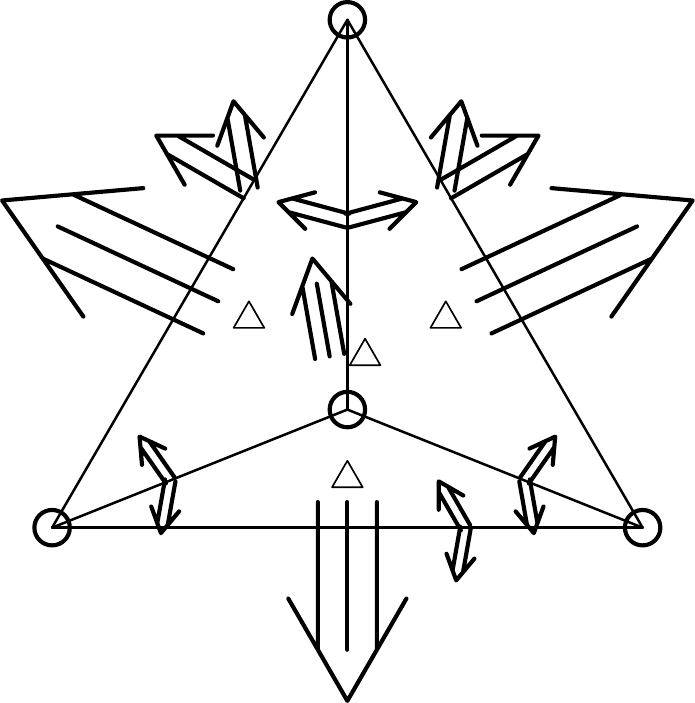}
\\ \hline 
5 & 
\includegraphics[width=1.1in]{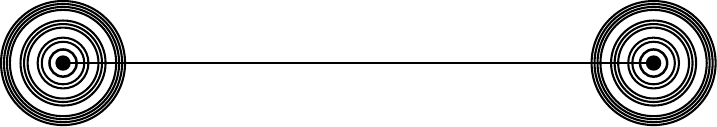}
&
\includegraphics[width=1.1in]{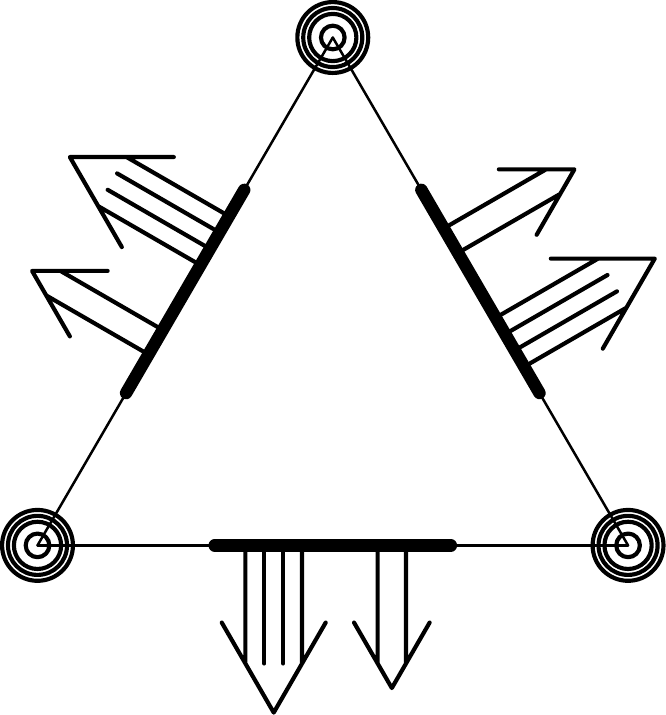}
&
\includegraphics[width=1.1in]{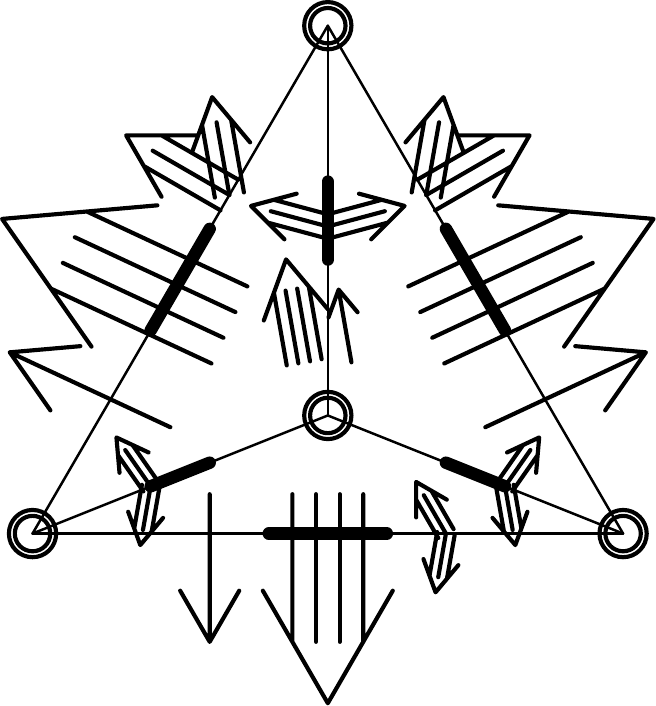}
\\ \hline 
\end{tabular}
\label{tab:DOF}
\end{center}
\end{table}
 
  \begin{lemma}[number of degrees of freedom] \label{lm:number-dof} 
  For any $m,n \geq 1$, it holds that 
  \begin{equation} \label{eq:number-dofl}
  \# D_{T,\ell}^{(m,n)} = \left\{
  \begin{aligned}
  \dim \mc P_{m - \ell n}(T) - \dim\mc P_{m - \ell n - n-1}(T) &\quad \text{for } 0 \leq \ell < \lceil \frac{m}{n} \rceil - 1, \\
  \dim \mathcal{P}_{m-\ell n}(T) &\quad \text{for } \ell = \lceil \frac{m}{n} \rceil - 1.
  \end{aligned}
  \right.
  \end{equation}
  As a consequence, by summing and reordering the above, we arrive at
  \begin{equation} \label{eq:number-dof}
   \#D_T^{(m,n)} = \dim \mc P_m(T) + \sum_{\ell = 1}^{\lceil \frac{m}{n} \rceil - 1} \left(\dim \mc P_{m - \ell n}(T) - \dim\mc P_{m - \ell n-1}(T)\right).
  \end{equation}
  \end{lemma}
  
  \begin{proof}
  Let \(\binom{a}{b}\) denote the binomial coefficient defined as 
\[
\binom{a}{b} := 
\begin{cases} 
\frac{a!}{b!(a-b)!} & \text{if } a \geq b, \\
0 & \text{otherwise}.
\end{cases}
\]
 Note that $\#\mathcal{F}_{T,k} = \binom{n+1}{n-k+1} = \binom{n+1}{k}$, and $\#\{\balpha \in A_k, |\balpha| = m - \ell n -k\} = \binom{m-\ell n - 1}{k-1}$, we count the number the degrees of freedom in following two cases.
 \begin{enumerate}
 \item Case 1: $0 \leq \ell < \lceil \frac{m}{n} \rceil - 1$. Then,
$$
\begin{aligned}
\#D_{T,\ell}^{(m,n)} &= \sum_{k=1}^n\binom{n+1}{n-k+1} \binom{m-\ell n - 1}{k-1} \\
&= \sum_{k'=0}^{\infty} \binom{n+1}{n-k'} \binom{m-\ell n - 1}{k'} - \binom{m-\ell n - 1}{n} \\
&= \binom{m-\ell n + n}{n} - \binom{m-\ell n - 1}{n} \\
&= \dim \mc P_{m - \ell n}(T) - \dim\mc P_{m - \ell n - n-1}(T).
\end{aligned}
 $$ 
\item Case 2: $\ell = \lceil \frac{m}{n} \rceil - 1$. Then, 
$$
\begin{aligned}
\#D_{T,\ell}^{(m,n)} &= \sum_{k=1}^{m-\ell n}\binom{n+1}{n-k+1} \binom{m-\ell n - 1}{k-1} \\
&= \sum_{k'=0}^{\infty} \binom{n+1}{n-k'} \binom{m-\ell n - 1}{k'} \\
&= \binom{m-\ell n + n}{n}  = \dim \mc P_{m - \ell n}(T).
\end{aligned}
 $$ 
 \end{enumerate} 
 Thus, the proof of \eqref{eq:number-dofl} is complete.
\end{proof}

Analogous to Lemma \ref{lm:MWX-derivative}, the degrees of freedom \eqref{eq:DOF} satisfy the following crucial property.
The proof essentially involves a recursive application of Green's lemma. For completeness, we provide a sketch of the proof; the remaining details are identical to those in the proof of \cite[Lemma 2.1]{wang2013minimal}.

\begin{lemma}[equivalence for vanishing DOF]  \label{lm:LW-derivative}
Let $0 \leq \ell \leq \lceil \frac{m}{n} \rceil -1$ and $D_{T,\ell}^{(m,n)}$ is defined in \eqref{eq:DOF-layer}. Then, for any $v \in H^m(T)$, the vanishing of all degrees of freedom in $D_{T,\ell}^{(m,n)}$ is equivalent to
    \begin{equation} \label{eq:LW-derivative-l}
      \frac{1}{|F|}\int_F \nabla^{m-\ell n - k} v = 0
       \quad \forall F\in \mc F_{T,k}, \forall 0\leq k \leq \min\{n,m-\ell n\},
    \end{equation}
    where $\nabla^j$ is the $j$th Hessian tensor for any integer $j \geq 0$. As a consequence, for any $v \in H^m(T)$, the vanishing of all degrees of freedom in $D_{T}^{(m,n)}$ is equivalent to
        \begin{equation} \label{eq:LW-derivative}
    \frac{1}{|F|}\int_F \nabla^{m-\ell n-k}v = 0 \quad \forall F\in \mc F_{T,k}, \forall 0\leq k \leq n, m-\ell n - k \geq 0.
    \end{equation}
\end{lemma}
\begin{proof}[Sketch of proof] When $0 \leq \ell < \lceil \frac{m}{n} \rceil - 1, k=n$ or $\ell = \lceil \frac{m}{n} \rceil - 1, k = m - \ell n$, the lemma is obvious true, since in the former case $F$ is a vertex, and in the latter case it only involves the average of the function values on $F$.

Assume that it is true for all $k = i+1, \ldots, \min\{n,m-\ell n\}$. We consider the case in which $k = i$. For any $(n-k)$-dimensional simplex $F$ and $|\balpha| = m - \ell n - k$, the proof is then divided into two cases. In the first case, when $\alpha_{k+1} = \cdots = \alpha_n = 0$, we have $\frac{1}{|F|}\int_F \partial^{\balpha}v = d_{T,F,\balpha}(v) = 0$. In the second case, where $\partial^{\balpha}v$ includes at least one tangential derivative, we apply Green's lemma and the inductive hypothesis to complete the proof.
\end{proof}

\subsection{Shape function space}
As introduced in the MWX elements, the convergence of nonconforming elements is closely linked to weak continuity and the design of degrees of freedom. However, a core challenge in constructing such element lies in finding shape function space that align with the degrees of freedom. We begin by presenting the construction of the shape function space
\begin{equation}  \label{eq:shape}
    P_T := P_T^{(m,n)} := \sum_{\ell=0}^{ \lceil \frac{m}{n} \rceil - 1} \lambda_1^{\ell(n+1)} \mathcal P_{m-\ell n}(T).
\end{equation}
Due to the inherent symmetry, in the shape function space \eqref{eq:shape}, the barycentric coordinate $\lambda_1$ can be consistently replaced by any $\lambda_i$ for $i=0, \dots, n$ in each term of the summation. For simplicity, we choose to work with $\lambda_1$.

\begin{remark}[1D case]
From \eqref{eq:shape}, it is easy to see that when $n=1$,
$P_T^{(m,1)} =  \sum_{\ell=0}^{m-1} \lambda_1^{2\ell} \mathcal P_{m- \ell}(T) = \mathcal{P}_{2m-1}(T)$.
At this point, the degrees of freedom given by \eqref{eq:DOF} correspond to the values at the two vertices, along with the values of the first derivative up to the $(m-1)$-th derivative. That is, the elements become the standard one-dimensional $C^{m-1}$-conforming elements.
\end{remark}

\begin{lemma}[dimension of $P_T^{(m,n)}$] \label{lm:dim-PT}
It holds that
\begin{equation} \label{eq:dim-PT}
 \dim P_T^{(m,n)} = \dim \mc P_m(T) + \sum_{\ell = 1}^{\lceil \frac{m}{n} \rceil - 1} \left(\dim \mc P_{m - \ell n}(T) - \dim\mc P_{m - \ell n-1}(T)\right),
\end{equation}
which shows that $\dim P_T^{(m,n)} = \# D_T^{(m,n)}$ by using \eqref{eq:number-dof}.
\end{lemma}
\begin{proof}
Let $V_t := \sum_{\ell = 0}^t \lambda_1^{\ell(n+1)} \mathcal P_{m-\ell n}(T)$. We first show
\begin{equation} \label{eq:Vt}
V_{t-1}\cap \lambda_1^{t(n+1)}\mc P_{m-tn}(T) = \lambda_1^{t(n+1)}\mc P_{m-tn-1}(T) \quad \forall t \geq 1.
\end{equation}
On one hand, for any $p \in \mc P_{m-tn-1}(T)$, 
$$
\lambda_1^{t(n+1)}p = \lambda_1^{(t-1)(n+1)}\cdot \lambda_1^{n+1}p \in \lambda_1^{(t-1)(n+1)}\mc P_{m-(t-1)n}(T)\subset V_{t-1},
$$
whence $\lambda_1^{t(n+1)}\mc P_{m-tn-1}(T)\subset V_{t-1}\cap \lambda_1^{t(n+1)}\mc P_{m-tn}(T)$.

On the other hand, for any $p\in \mc P_{m-tn}(T)$ such that $\lambda_1^{t(n+1)}p \in V_{t-1}$, the fact that $V_{t-1} \subset \mathcal{P}_{m+t-1}(T)$ implies that $p \in \mathcal{P}_{m-tn-1}(T)$. Therefore, $\lambda_1^{t(n+1)}p \in \lambda_1^{t(n+1)}\mc P_{m-tn-1}(T)$, and hence $V_{t-1}\cap \lambda_1^{t(n+1)}\mc P_{m-tn}(T) \subset \lambda_1^{t(n+1)}\mc P_{m-tn-1}(T)$. This proves \eqref{eq:Vt}.

Using \eqref{eq:Vt}, and noticing that $V_t = V_{t-1} + \lambda_1^{t(n+1)}\mc P_{m-tn}(T)$, we have 
$$
\begin{aligned}
\dim V_{t} &= \dim V_{t-1} + \dim \mc P_{m-tn}(T) - \dim\left(V_{t-1}\cap \lambda_1^{t(n+1)}\mc P_{m-tn}(T) \right), \\
&=  \dim V_{t-1} + \dim \mc P_{m-tn}(T) - \dim\mc P_{m-tn-1}(T), \quad \forall t \geq 1.
\end{aligned}
$$    
This, together with the fact that $\dim V_0 = \dim \mathcal{P}_m(T)$, leads to \eqref{eq:dim-PT}.
\end{proof} 


\subsection{Space representation with integral terms}
The shape function space \eqref{eq:shape} has an equivalent integral-type representation, which is useful in proving its unisolvence.
We firstly introduce the following notation to simplify the representation of integral on one variable.

\begin{definition}[integral on $\blam$]
For a polynomial $f$ defined on $T$, let $f = f(\lambda_1,\ldots,\lambda_n) := f(\blam)$, define the integral
\begin{equation}
  (I_{0,\lambda_i}f)(\blam) := \int_0^{\lambda_i}f(\lambda_1,\ldots,\lambda_{i-1},s,\lambda_{i+1},\ldots,\lambda_n)\,\mathrm{d} s.
\end{equation} 
To write the multiple integral in a more compact form, we introduce the following notation:
\begin{equation}
  (I_{0,\blam}^{\balpha}f)(\blam):=I_{0,\lambda_1}^{\alpha_1}I_{0,\lambda_2}^{\alpha_2} \cdots I_{0,\lambda_n}^{\alpha_n}f(\blam) \quad \forall\balpha=(\alpha_1,\alpha_2,\ldots,\alpha_n),\alpha_i\geq 0.
\end{equation}
\end{definition}

\begin{lemma}[equivalent integral-type representation]     \label{lm:PT-integral}
Let $g_\ell$ be any single variable polynomial with degree $\ell(n+1)$. Then, for any integer $t\geq 0$, we have
 \begin{equation} \label{eq:PT-integral}
     \sum_{\ell = 0}^t \lambda_1^{\ell(n+1)} \mc P_{m-\ell n}(T) = 
     \mathcal{P}_m(T) + \sum_{\ell =1}^t \left\langle I_{0,\blam}^\balpha g_\ell(\lambda_1): |\balpha| = m - \ell n \right\rangle.
    \end{equation}
  \end{lemma}
  \begin{proof}
    We prove it by induction on $t$. For $t=0$, the lemma is trivial. Suppose the lemma is true for $t-1$, we will prove it for $t$. 
    Let $g_t(\lambda_1) = \sum_{i=0}^{t(n+1)} c_i \lambda_1^i$ with $c_{t(n+1)} \neq 0$. For any multi-index $\bbeta$, we have 
    $$
    I_{0,\blam}^\bbeta g_t(\lambda_1) = \sum_{i=0}^{t(n+1)} \tilde{c}_i \lambda_1^{i+\beta_1}\lambda_2^{\beta_2} \cdots \lambda_n^{\beta_n},
    $$
    where $\tilde{c}_i$ are some coefficients and obviously $\tilde{c}_{t(n+1)} \neq 0$.
  
    {\it Step 1.} For any $|\balpha|=m-tn$, we then consider the term $\lambda_1^{i+\alpha_1}\lambda_2^{\alpha_2} \cdots \lambda_n^{\alpha_n}$ in two cases.
    \begin{enumerate}
    \item Case 1: $i \leq tn$. Obviously, $\lambda_1^{i+\alpha_1}\lambda_2^{\alpha_2} \cdots \lambda_n^{\alpha_n} \in \mathcal{P}_m(T)$.
    \item Case 2: $tn<i\leq t(n+1)$, i.e., $0 < \ell_0 := i-tn \leq t$. Then, some elementary calculations show that 
    $$ 
    \begin{aligned}
    \ell_0(n+1) &= (i-tn)(n+1) \leq i \leq i + \alpha_1,\\
    i+|\balpha| - \ell_0(n+1) &= m - \ell_0 n. 
    \end{aligned}
    $$ 
    This implies $\lambda_1^{i+\alpha_1}\lambda_2^{\alpha_2} \cdots \lambda_n^{\alpha_n} \in \lambda_i^{\ell_0(n+1)} \mc P_{m-\ell_0 n}(T)$.
    \end{enumerate}
    Combining the above two cases, we have $\lambda_1^{i+\alpha_1}\lambda_2^{\alpha_2} \cdots \lambda_n^{\alpha_n} \in \sum_{\ell = 0}^t \lambda_1^{\ell(n+1)} \mc P_{m-\ell n}(T)$, so does $I_{0,\blam}^\balpha g_t(\lambda_1)$. Therefore, 
    $$ 
    \left \langle I_{0,\blam}^\balpha g_t(\lambda_1): |\balpha| = m-tn \right\rangle \subset \sum_{\ell = 0}^t \lambda_1^{\ell(n+1)} \mc P_{m-\ell n}(T),
    $$
    which, together with the induction hypothesis, yields
    \begin{equation} \label{eq:PT-integral1}
    \mathcal{P}_m(T) + \sum_{\ell =1}^t \left\langle I_{0,\blam}^\balpha g_\ell(\lambda_1): |\balpha| = m - \ell n \right\rangle \subset \sum_{\ell = 0}^t \lambda_1^{\ell(n+1)} \mc P_{m-\ell n}(T).
    \end{equation}
  
  {\it Step 2.} We intend to show that for any $\theta \leq t(n+1)$, $|\bbeta| \leq m - tn$, 
    \begin{equation} \label{eq:PT-integral-result}
    p:=\lambda_1^\theta \lambda_1^{\beta_1}...\lambda_n^{\beta_n} \in \mathcal{P}_m(T) + \sum_{\ell =1}^t \left\langle I_{0,\blam}^\balpha g_\ell(\lambda_1): |\balpha| = m - \ell n \right\rangle,
    \end{equation}
    The proof of \eqref{eq:PT-integral-result} can be done by considering the following three cases.
    
    \begin{enumerate}
    \item Case 1: $\theta + |\bbeta| \leq m$. Obviously, $p \in \mathcal{P}_m(T)$.
    \item Case 2: $m < \theta + |\bbeta| < m+t$, i.e., $0 < \ell_0 := \theta + |\bbeta| - m < t$. Then, similar to the Case 2 of {\it Step 1}, we have 
    $$ 
    \begin{aligned}
    \ell_0(n+1) &= (\theta + |\bbeta| - m)(n+1) \leq \theta \leq \theta + \beta_1, \\
    \theta + |\bbeta| - \ell_0(n+1) &= m - \ell_0 n.
    \end{aligned}
    $$
    Therefore, by induction hypothesis, we have 
    $$ 
    \begin{aligned}
    p \in \lambda_i^{\ell_0(n+1)} \mc P_{m-\ell_0 n}(T) & \subset \sum_{\ell=0}^{t-1}\lambda_i^{\ell(n+1)} \mc P_{m-\ell n}(T) \\
    & \subset \mathcal{P}_m(T) + \sum_{\ell =1}^{t-1} \left\langle I_{0,\blam}^\balpha g_\ell(\lambda_1): |\balpha| = m - \ell n \right\rangle.
    \end{aligned}
    $$  
    \item Case 3: $\theta = t(n+1)$ and $|\bbeta|=m-tn$. In this case, taking $\balpha = \bbeta$, we have 
    $$
    \begin{aligned}
    \frac{1}{\tilde{c}_{t(n+1)}} I_{0,\blam}^{\balpha} g_t(\lambda_1) - p &= \sum_{i=0}^{t(n+1)-1} \frac{\tilde{c}_{i}}{\tilde{c}_{t(n+1)}} \lambda_i^{i+\beta_1} \lambda_2^{\beta_2} \cdots \lambda_n^{\beta_n} \\
    &\in \mathcal{P}_m(T) + \sum_{\ell =1}^{t-1} \left\langle I_{0,\blam}^\balpha g_\ell(\lambda_1): |\balpha| = m - \ell n \right\rangle.
    \end{aligned}
    $$ 
    Here, the last step comes from the fact that all the terms belong to either Case 1 or Case 2. Hence, \eqref{eq:PT-integral-result} holds for this case.
    \end{enumerate}
    Using \eqref{eq:PT-integral-result}, we have 
    \begin{equation} \label{eq:PT-integral2}
    \sum_{\ell = 0}^t \lambda_1^{\ell(n+1)} \mc P_{m-\ell n}(T) \subset
    \mathcal{P}_m(T) + \sum_{\ell =1}^t \left\langle I_{0,\blam}^\balpha g_\ell(\lambda_1): |\balpha| = m - \ell n \right\rangle. 
    \end{equation}
    The induction argument then completes by combining \eqref{eq:PT-integral1} and \eqref{eq:PT-integral2}.
  \end{proof}

  \begin{corollary}[general integral-type representation of $P_T^{(m,n)}$] \label{co:integral-PT}
    Let $g_\ell$ be any polynomial with degree $\ell(n+1)$. Then, we have
    \begin{equation} \label{eq:integral-PT}
      P_T^{(m,n)} = 
      \mathcal{P}_m(T) + \sum_{\ell =1}^{ \lceil \frac{m}{n} \rceil - 1} \left\langle I_{0,\blam}^\balpha g_\ell(\lambda_1): |\balpha| = m - \ell n \right\rangle.
    \end{equation}
  \end{corollary}
    From this corollary we see that the shape function space can be represented by the integral form with {\it generator} $g_\ell$. The proper choice of the $g_\ell$ will facilitate the proof of unisolvence, as we will demonstrate in the next section. 

\section{Unisolvence} \label{sec:unisolvence}
This section discusses the unisolvence of the proposed finite elements, using a special integral-type representation of $P_T^{(m,n)}$
by using nonconforming bubbles from single-variable composition. For a single-variable function $f(x)$, we denote its $k$th order derivative by $f^{(k)}(x)$.

\begin{definition}[nonconforming bubble] \label{df:bubble}
Given a layer $\ell$, a function $b$ is called nonconforming bubble of
  layer $\ell$ on element $T$ if it satisfies 
$$ 
d_{T,j}(b)=0,\quad \forall d_{T,j}\in D_T^{(\ell n,n)},
$$
where we recall the degrees of freedom given in \eqref{eq:DOF}.
\end{definition}

It is evident that the nonconforming bubble in Definition \ref{df:bubble} is not unique. In the following, we will explore a specific type of nonconforming bubbles formed by the composition of single-variable functions and $\lambda_1$. These particular single-variable functions can be selected as generators in \eqref{eq:integral-PT}, leading to a specific integral-type representation of $P_T^{(m,n)}$.


\subsection{Nonconforming bubbles from single-variable composition}
We first present a unisolvence lemma for a single-variable function.

\begin{lemma}[unisolvence for single-variable function]   \label{lm:1d-unisolvence}
For any $s,t \in \mathbb{N}$ with $s + t \geq 1$, a polynomial $v \in \mathcal{P}_{s+t-1}(\mathbb{R})$ is uniquely determined by the following values:
  $$ 
  \begin{aligned}
  d_{i,0}(v) &:= v^{(i)}(0) \quad  i=0,1,\ldots,s-1, \\
  d_{j,1}(v) &:= v^{(\alpha_j)}(1) \quad j=0,1,\ldots,t-1,
  \end{aligned}
  $$
  where $\alpha_j \in \mathbb{N}$ satisfying $0\leq \alpha_0 < \alpha_1 < \cdots <\alpha_{t-1} \leq s+t-1$.
\end{lemma}
\begin{proof}
  We just need to verify that $v = 0$ when the degrees of freedom vanishes, since both the number degrees of freedom and dimension of $\mathcal{P}_{s+t-1}(\mathbb{R})$ are $s+t$. We prove it by induction on $s+t$. 
  
The result is trivial when $s+t=1$. Assuming it holds for $s+t-1$, we now consider the situation when $s+t$ and discuss it in the following four cases.
  
 \begin{enumerate}
 \item Case 1: $t=0$. Since $v\in \mathcal{P}_{s-1}(\mathbb{R})$ and $v^{(i)}(0)=0$ for $i=0,1,\ldots,s-1$, we easily see that $v=0$.
 \item Case 2: $s=0$. In this case, we must have $\alpha_j = j$ for $j=0,\ldots,t-1$ and hence $v = 0$.
 \item Case 3: $s,t>0$, $\alpha_0 > 0$. Consider $w=v'$, then $w\in \mathcal{P}_{s+t-2}(\mathbb{R})$ satisfying
  $$ 
  \begin{aligned}
  w^{(i)}(0) &= 0 \quad i=0,1,\ldots,s-2, \\
  w^{(\alpha_j-1)}(1) &= 0 \quad j = 0,1,\ldots,t-1.
  \end{aligned}
  $$
  Then, the induction hypothesis implies that $v'=w=0$. Notice that $s > 0$, we have $v(0) = 0$, whence $v = 0$.
 \item Case 4: $s,t>0$, $\alpha_0 = 0$, i.e., $v(1) = 0$. Then, there exists $\tilde{v} \in \mathcal{P}_{s+t-2}(\mathbb{R})$ such that $v = (1-x)\tilde{v}(x)$. Since $v^{(k)}(x) = (1-x)\tilde{v}^{(k)}(x) - k\tilde{v}^{(k-1)}(x)$, we have 
 $$
 \tilde{v}^{(\alpha_j -1)}(1) = -\frac{1}{\alpha_j}v^{(\alpha_j)} = 0 \quad j = 1, \ldots, t-1.
 $$ 
 Further, applying $v^{(i)}(0) = 0$ for $i=0,\ldots,s-1$ recursively yields 
 $$ 
 \tilde{v}^{(i)}(0) = 0 \quad i=0,\ldots,s-1.
 $$ 
 The induction hypothesis then implies that $\tilde{v} = 0$, so does $v$.
 \end{enumerate}
Combining the above cases, we complete the inductive proof.
\end{proof}

We are in the position to show the existence of the nonconforming bubbles.
\begin{theorem}[nonconforming bubbles from single-variable composition] \label{tm:bl}
  There exists a single-variable polynomial $b_\ell \in \mc P_{\ell(n+1)}(\mathbb{R})$ such that 
    \begin{equation} \label{eq:bl}
    \begin{aligned}
      b_\ell^{(i)}(0) &= 0,\quad i=0,1,...,\ell n-1,\\
      b_\ell^{(j)}(1) &= 0,\quad j=0,n,2n,...,(\ell - 1)n,
    \end{aligned}
  \end{equation}
 and $b_\ell(\lambda_1)$ is a nonconforming bubble of layer $\ell$ (see Definition \ref{df:bubble}). Moreover, if $b_\ell$ is monic, then it is unique.
\end{theorem}
\begin{proof}
  By Lemma \ref{lm:1d-unisolvence} (unisolvence for single-variable function), we know that there exists a unique polynomial $p(x) \in \mathcal{P}_{\ell(n+1)-1}(\mathbb{R})$ satisfying
$$
    \begin{aligned}
      p^{(i)}(0) &= [x^{\ell(n+1)}]^{(i)}(0),\quad i=0,1, \ldots, \ell n-1,\\
      p^{(j)}(1) &= [x^{\ell(n+1)}]^{(j)}(1),\quad j=0,n,2n, \ldots, (\ell-1)n.
    \end{aligned}
$$
Let $b_\ell(x) := x^{\ell(n+1)}-p(x)$ which satisfies \eqref{eq:bl}. Next, we show that $b_\ell(\lambda_1)$ is a nonconforming bubble of layer $\ell$, which, thanks to Lemma \ref{lm:LW-derivative} (equivalence for vanishing DOF), amounts to showing
\begin{equation} \label{eq:bl-equiv}
    \frac{1}{|F|}\int_F \nabla^{(\ell - t)n-k} b_\ell(\lambda_1) = 0 \quad  \forall F\in \mc F_{T,k},~\forall 1\leq k\leq n,~\forall 0\leq t\leq \ell-1.
\end{equation}
Since $b_\ell$ is single variable, we therefore only need to check the derivatives of $b_\ell(\lambda_1)$ with respect to $\lambda_1$ in \eqref{eq:bl-equiv}. 

{\it Step 1: verification of \eqref{eq:bl-equiv}}. For any $F \in \mc F_{T,n}$ (i.e., $F$ is a vertex), either $\lambda_1 = 1$ ($F = a_1$) or $\lambda_1 = 0$ (other vertex), \eqref{eq:bl} implies that 
$$
b_\ell^{((\ell-t-1)n)}(0) = b_\ell^{((\ell-t-1)n)}(1)=0 \quad \forall 0\leq t\leq \ell-1.
$$
Therefore, \eqref{eq:bl-equiv} holds for $k = n$.

For any $F \in \mathcal{F}_{T,k}$ with $1\leq k \leq n-1$, we consider the following two cases.

\begin{enumerate}
\item Case 1:  $a_1\notin F$. In this case, $\lambda_1|_F = 0$ and \eqref{eq:bl} directly imply 
  $$
   \frac{1}{|F|} \int_F b_\ell^{((\ell-t)n-k)}(\lambda_1) = b_\ell^{((\ell-t)n-k)}(0) = 0 \quad \forall 0\leq t\leq \ell-1.
   $$
\item Case 2: $a_1\in F$. Denote $\hat{T}_{n-k}$ the reference $(n-k)$-simplex. Using the linear transformation from $\hat{T}$ to $F$, where $a_1$ is the mapped from $(1,0,\cdots,0)$, we have that, for any $0 \leq t \leq \ell-1$,
$$
  \begin{aligned}
    & \frac{1}{|F|} \int_F b_\ell^{((\ell-t)n-k)}(\lambda_1) \,\mathrm{d}{\boldsymbol x} = \frac{1}{|\hat{T}_{n-k}|} \int_{\hat{T}_{n-k}} b_\ell^{((\ell-t)n-k)}(\hat{x}_1) \,\mathrm{d}\hat{\boldsymbol x}  \\
    =~ & (n-k)! \int_0^1\int_0^{1-\hat{x}_{n-k}} \cdots \int_0^{1-\sum_{i=2}^{n-k}\hat{x}_i} b_\ell^{((\ell-t)n-k)}(\hat{x}_1) \,\mathrm{d}\hat{\boldsymbol x}  \\
    =~& (n-k)! \int_0^1\int_0^{1-\hat{x}_{n-k}} \cdots \int_0^{1-\sum_{i=3}^{n-k}\hat{x}_i} 
    \Big(b_\ell^{((\ell-t)n-k-1)}(1-\sum_{i=2}^{n-k}\hat{x}_i) \\
     & \qquad\qquad\qquad\qquad\qquad\qquad\quad  - b_\ell^{((\ell-t)n-k-1)}(0)\Big) \,\mathrm{d} \hat{x}_2 \cdots \mathrm{d} \hat{x}_{n-k} \\
    =~& (n-k)! \left(b_\ell^{((\ell-t-1)n)}(1) - b_\ell^{((\ell-t-1)n)}(0)- \sum_{i=1}^{n-k-1} \frac{1}{i!}b_\ell^{((\ell-t-1)n+i)}(0) \right).
  \end{aligned}
  $$
Based on \eqref{eq:bl}, each term in the final step of the equation vanishes, ensuring that \eqref{eq:bl-equiv} is valid for $1 \leq k \leq n-1$.
\end{enumerate}
Combining the above two cases, the existence of single variable nonconforming bubble is shown.

  {\it Step 2: uniqueness}. Thanks to Lemma \ref{lm:LW-derivative} (equivalence for vanishing DOF) and the above calculation, we obtain 
  $$
  \begin{aligned}
&  b_\ell(\lambda_1) \text{ is a nonconforming bubble } \Leftrightarrow \eqref{eq:bl-equiv} \\
\Leftrightarrow & \left\{
\begin{aligned}
b_\ell^{((\ell-t-1)n)}(1) = 0, &\qquad \forall 0 \leq t \leq \ell-1, \\
b_\ell^{((\ell-t)n-k)}(0) = 0, &\quad  \forall 1 \leq k \leq n, \forall 0 \leq t \leq \ell-1,  \\
\begin{aligned}
& b_\ell^{((\ell-t-1)n)}(1) - b_\ell^{((\ell-t-1)n)}(0) \\
&- \sum_{i=1}^{n-k-1} \frac{1}{i!}b_\ell^{((\ell-t-1)n+i)}(0)
\end{aligned}=0, &\quad \forall 1 \leq k \leq n-1, \forall 0 \leq t \leq \ell-1,
\end{aligned}
\right. \\
\Leftrightarrow~ & \eqref{eq:bl}.
\end{aligned}
$$ 
Therefore, the uniqueness of $b_\ell$ (in the monic sense) follows from the uniqueness of $p(x)$, as stated in Lemma \ref{lm:1d-unisolvence} (unisolvence of single-variable function).
\end{proof}

In Corollary \ref{co:integral-PT} (general integral-type representation of $P_T^{(m,n)}$), by selecting the generator as the $b_\ell$ constructed in Theorem \ref{tm:bl}, we obtain the spatial integral-type representation used in the actual proof of unisolvence:
    \begin{equation} \label{eq:PT-bl}
     P_T^{(m,n)} = 
      \mathcal{P}_m(T) + \sum_{\ell =1}^{ \lceil \frac{m}{n} \rceil - 1} \left\langle I_{0,\blam}^\balpha b_\ell(\lambda_1): |\balpha| = m - \ell n \right\rangle.
    \end{equation}
 Before discussing the proof in detail, let us first outline the advantages of this special representation when dealing with the induction argument in $m$.

First, if $m-1$ is not a multiple of $n$, the number of summands in \eqref{eq:PT-bl}, i.e., $\lceil \frac{m}{n}\rceil$, remains the same for both $m-1$ and $m$. In this case, since differentiation is the inverse operation of integration, it follows directly that many first-order derivatives of functions in $P_T^{(m,n)}$ are contained within $P_T^{(m-1,n)}$. This allows for the direct application of the inductive hypothesis (see Case 1 in the proof of Theorem \ref{tm:unisolvence}).

However, when $m-1$ is a multiple of $n$, a new layer of space generated by a nonconforming bubble is introduced in \eqref{eq:PT-bl} from $m-1$ to $m$. Essentially, the nonconforming bubble is obtained by composing $\lambda_1$ to a one-dimensional function $b_\ell$. To handle these terms in the proof of unisolvence, we additionally require a property of the $b_\ell$ stated as follows.

\begin{lemma}[a property of $b_\ell$] \label{lm:bl-prop}
The monic polynomial $b_\ell \in \mathcal{P}_{\ell(n+1)}(\mathbb{R})$ defined in \eqref{eq:bl} has the following property: 
\begin{equation} \label{eq:bl-prop}
\int_0^1 (1-t)^{n-1} b_{\ell}(t) \,\mathrm{d} t \neq 0 \quad \forall n \geq 1, \ell \geq 1.
\end{equation}
\end{lemma}
\begin{proof} We first conduct integration by parts $(n-1)$ times to have 
$$
\begin{aligned}
\int_0^1 (1-t)^{n-1} b_{\ell}(t) \,\mathrm{d} t &= (n-1) \int_0^1 (1-t)^{n-2} \int_0^{t} b_{\ell}(t_2) \,\mathrm{d}t_2 \mathrm{d}t \\
&= (n-1)! \int_0^1 \int_0^{t_1} \cdots \int_0^{t_{n-1}} b_\ell(t_{n}) \,\mathrm{d}t_{n}\cdots \mathrm{d}t_1.
\end{aligned}
$$ 
Define 
$$
\Psi_\ell(x) := \int_0^x \int_0^{t_1} \cdots \int_0^{t_{n-1}} b_\ell(t_{n}) \,\mathrm{d}t_{n}\cdots \mathrm{d}t_1.
$$
Since $b_\ell \in \mathcal{P}_{\ell n+\ell}(\mathbb{R})$, it follows that $\Psi_\ell \in \mathcal{P}_{\ell n+\ell+n}(\mathbb{R})$. Assume that \eqref{eq:bl-prop} does not hold, i.e., $\Psi_\ell(1) = 0$. In light of the definition of $b_\ell$ in \eqref{eq:bl} and the fact that $\Psi_\ell^{(n)}(x)= b_\ell(x)$, we easily see that 
$$ 
\begin{aligned}
      \Psi_\ell^{(i)}(0) = 0, \quad &\forall i=0,1, \ldots,\ell n+n-1, \\
      \Psi_\ell^{(j)}(1) = 0, \quad &\forall j=0, n , 2n, \ldots, \ell n.
\end{aligned}
$$
Now, applying Lemma \ref{lm:1d-unisolvence} (unisolvence for single-variable function) leads to $\Psi(x) = 0$, which contradicts with $\Psi^{(n)}(x) = b_\ell(x) \neq 0$. Therefore, $\Psi_\ell(1) \neq 0$, which means that \eqref{eq:bl-prop} holds.
\end{proof}

\subsection{Proof of unisolvence} Building upon the previous preparations, we will conclude this section with the proof of unisolvence.

\begin{theorem}[unisolvence] \label{tm:unisolvence}
For any $m,n \geq 1$, $D_T^{(m,n)}$ is $P_T^{(m,n)}$-unisolvent.
  \end{theorem}
  \begin{proof}
As shown in Lemma \ref{lm:dim-PT}, dimension of $P_T^{(m,n)}$ matches the number of local degrees of freedom, it suffices to show that $v = 0$ if all the degrees of freedom vanish.
    
For any $n \geq 1$, we conduct the induction argument on $m$. Clearly, the case in which $m = 1$ is standard (i.e., Crouzeix-Raviart elements). Suppose the statement is true for all the case less than $m$, we will prove it for the case $m$. We write $m = \ell_0 n + r$ where $\ell_0 := \lceil \frac{m}{n}\rceil - 1$ and $1\leq r\leq n$.
Consider $v_i:=\partial_{\lambda_i}v~(1 \leq i \leq n)$. Lemma \ref{lm:LW-derivative} (equivalence for vanishing DOF) implies that 
\begin{equation} \label{eq:dof-induction}
\begin{aligned}
& d_{T,j}(v) = 0 \quad \forall d_{T,j} \in D_T^{(m,n)} \\
\Leftrightarrow~ & \frac{1}{|F|}\int_F \nabla^{m-\ell n-k}v = 0 \quad \forall F\in \mc F_{T,k}, \forall 0\leq k \leq n, m-\ell n - k \geq 0,  \\
 \Rightarrow~ & \frac{1}{|F|} \int_F \nabla^{m-1-\ell n-k}v_i = 0 \quad \forall F\in \mc F_{T,k}, \forall 0\leq k \leq n, m -1 -\ell n - k\geq 0, \\
\Leftrightarrow ~& d_{T,j}(v_i) = 0 \quad \forall d_{T,j} \in D_T^{(m-1,n)}.
\end{aligned}
\end{equation}

{\it Case 1: $r>1$}. In this case, $m-1 = \ell_0 n + (r-1)$, where $1 \leq r-1 \leq n$. That is, $\lceil \frac{m-1}{n}\rceil - 1 = \ell_0$. 
First, for $i \neq 1$, in light of the integral representation of $P_T^{(m,n)}$ in \eqref{eq:PT-bl}, we have for any $1\leq \ell \leq \ell_0$ and $|\balpha| = m - \ell n$,
     $$ 
     \partial_{\lambda_i} I_{0,\blam}^{\balpha} b_\ell(\lambda_1) =\left\{
     \begin{aligned}
      I_{0,\blam}^{\balpha - \boldsymbol{e}_i} b_\ell(\lambda_1), &\quad \text{ if } \alpha_i \geq 1,\\
      0, &\quad \text{ if } \alpha_i = 0,
     \end{aligned}\right.   
     \in \left\langle I_{0,\blam}^\bbeta b_\ell(\lambda_1): |\bbeta| = m -1 - \ell n \right\rangle. 
     $$
     where $\boldsymbol{e}_i$ is the unit vector in $i$-th entry. This implies that $v_i\in P_T^{(m-1,n)}$, and by \eqref{eq:dof-induction} and induction hypothesis, we know that $v_i=0~(i=2,3,\ldots,n)$, whence $v$ only depends on $\lambda_1$. 
     
     In light of \eqref{eq:shape}, there exists $p \in \mathcal{P}_{m+\ell_0}(\mathbb{R})$ such that $v = p(\lambda_1)$. So $\partial_{\lambda_1}v = p'(\lambda_1)\in P_T^{(m-1,n)}$. By induction hypothesis, $v_1=0$, thus $v$ is a constant. Since $\int_F v=0$ for $F\in \mc F_{T,r}$, 
     we get $v=0$.
  
  {\it Case 2: $r = 1$}. In this case, $m - 1 = (\ell_0-1) n + n$. From \eqref{eq:PT-bl}, we recall that  
     $$ 
     \begin{aligned}
     P_T^{(m-1,n)} = & \mc P_{\ell_0 n}(T) + \sum_{\ell=1}^{\ell_0-1} \left\langle I_{0,\blam}^\balpha b_\ell(\lambda_1): |\balpha| = \ell_0n - \ell n \right\rangle, \\
     P_T^{(m,n)} = &\underbrace{\mc P_{\ell_0 n+1}(T) + \sum_{\ell=1}^{\ell_0-1} \left\langle I_{0,\blam}^\balpha b_\ell(\lambda_1): |\balpha| = \ell_0n - \ell n + 1 \right\rangle}_{:= \tilde{P}_T^{(\ell_0n+1,n)}} \\
     &\qquad \qquad\quad + \left\langle \int_0^{\lambda_1} b_{\ell_0}(t)\,\mathrm{d}t,  \lambda_2 b_{\ell_0}(\lambda_1), \cdots, \lambda_n b_{\ell_0}(\lambda_1) \right\rangle.
     \end{aligned}
     $$
     Therefore, there exist $\tilde{v} \in \tilde{P}_T^{(\ell_0n+1,n)}$ and coefficients $c_i~(1\leq i \leq n)$, such that 
     \begin{equation} \label{eq:unisolvence-v}
     v = \tilde{v} + c_1 \int_0^{\lambda_1} b_{\ell_0}(t)\,\mathrm{d}t + c_2 \lambda_2 b_{\ell_0}(\lambda_1) + \cdots + c_n \lambda_n b_{\ell_0}(\lambda_1).
     \end{equation}  
     The proof in Case 2 is then proceed in four steps.
     
     {\it Step 2-1: $\tilde{v}$ only depends on $\lambda_1$.} We consider $v_i := \partial_{\lambda_i}v$ for $2\leq i \leq n$, which yields
     $$ 
     0 = d_{T,j}(v_i) = d_{T,j}(\partial_{\lambda_i}\tilde{v} + c_i b_{\ell_0}(\lambda_1)) = d_{T,j}(\partial_{\lambda_i}\tilde{v}) \quad \forall d_{T,j} \in D_T^{(m-1,n)} = D_T^{(\ell_0n, n)}.
     $$ 
     Here, in the last step, we utilize the fact that $b_{\ell_0}(\lambda_1)$ is a nonconforming bubble of layer $\ell_0$ (see Theorem \ref{tm:bl}). Now, since $\tilde{v} \in \tilde{P}_T^{(\ell_0n+1,n)}$, we see that $\partial_{\lambda_i} \tilde{v} \in P_T^{(m-1,n)}$ by using a similar argument as Case 1 above. Then, the induction hypothesis yields $\partial_{\lambda_i} \tilde{v} = 0$ for $i = 2, \ldots, n$, whence $\tilde{v}$ only depends on $\lambda_1$. We could therefore write $\tilde{v} = \tilde{p}(\lambda_1)$ where $\tilde{p} \in \mathcal{P}_{\ell_0n+\ell_0}(\mathbb{R})$.
            
{\it Step 2-2: $c_i = 0$ for $i=2,\ldots,n$.} At this point, for any $(n-1)$-dimensional face $F \in \mathcal{F}_{T,1}$, we have $\frac{1}{|F|}\int_F v = 0$. For any $0 \leq j \leq n$, let $F_j$ denote the $(n-1)$-dimensional face that does not contain $a_j$. By symmetry, we have
\begin{equation}
\begin{aligned}
D_0 &:= \frac{1}{|F_j|}\int_{F_j} \tilde{v} = \frac{1}{|F_j|}\int_{F_j} \tilde{p}(\lambda_1) \\
D_1 &:= \frac{1}{|F_j|}\int_{F_j} \int_0^{\lambda_1} b_{\ell_0}(t)\,\mathrm{d}t 
\end{aligned}\quad \forall 0 \leq j \neq 1 \leq n,
\end{equation}
and
\begin{equation}
B := \frac{1}{|F_j|}\int_{F_j} \lambda_i b_{\ell_0}(\lambda_1) \quad \forall 0 \leq j\neq1 \leq n, 2\leq i\neq j \leq n.
\end{equation}
First, we check $ \frac{1}{|F_0|}\int_{F_0}v = 0 $, yielding $ D_0 + c_1D_1 + B \sum_{j=2}^n c_j = 0 $. Next, for $i = 2, \ldots, n$, we check $ \frac{1}{|F_i|} \int_{F_i} v = 0 $ and obtain $ D_0 + c_1D_1 + B \sum_{j=2, j\neq i}^n c_j = 0 $. Hence, we have $Bc_i = 0$ for $i = 2, \ldots, n$. Thanks to Lemma \ref{lm:bl-prop} (a property of $b_{\ell}$) and the standard linear transformation to the $(n-1)$-dimensional reference element $\hat{T}_{n-1}$, we have 
$$
\begin{aligned}
B &= (n-1)! \int_{\hat{T}_{n-1}} \hat{x}_2 b_{\ell_0}(\hat{x}_1) \,\mathrm{d}\hat{x}_1 \cdots \mathrm{d}\hat{x}_{n-1} \\
&= (n-1)(n-2) \int_{0}^1 b_{\ell_0}(\hat{x}_1) \int_{0}^{1-\hat{x}_1} \hat{x}_2 (1-\hat{x}_1 - \hat{x}_2)^{n-3} \,\mathrm{d}\hat{x}_2 \mathrm{d}\hat{x}_1 \\ 
&= \int_0^1 (1-\hat{x}_1)^{n-1}b_{\ell_0}(\hat{x}_1) \,\mathrm{d}\hat{x}_1 \neq 0,
\end{aligned}
$$ 
whence $c_i = 0$, $\forall 2 \leq i \leq n$.

{\it Step 2-3: $\tilde{v} = 0$.} Now, we have $v = \tilde{v} + c_1 \int_0^{\lambda_1} b_{\ell_0}(t) \,\mathrm{d}t = \tilde{p}(\lambda_1) + c_1 \int_0^{\lambda_1} b_{\ell_0}(t) \,\mathrm{d}t$. Then, applying \eqref{eq:dof-induction} to $v_1 :=\partial_{\lambda_1} v$, we have
$$ 
0 = d_{T,j}(v_1) = d_{T,j}(\tilde{p}'(\lambda_1) + c_1 b_{\ell_0}(\lambda_1)) = d_{T,j}(\tilde{p}'(\lambda_1))
\quad \forall d_{T,j} \in D_T^{(m-1,n)} = D_T^{(\ell_0n, n)}.
$$ 
Again, $\tilde{p}' \in \mathcal{P}_{\ell_0n+\ell_0-1}(\mathbb{R})$ so that $\tilde{p}' (\lambda_1) \in P_T^{(\ell_0n, n)} = P_T^{(m-1,n)}$. The induction hypothesis then yields $\tilde{p}'(\lambda_1) = 0$, whence $\tilde{v}$ is a constant $\tilde{c}$. Then, by checking $ \frac{1}{|F_1|}\int_{F_1}v = 0 $ and noting that $ \lambda_1|_{F_1} = 0 $, we obtain $ \tilde{c} = 0 $, which implies $ \tilde{v} = 0 $.

{\it Step 2-4: $c_1 = 0$}. Now go back to {\it Step 2-2}, where we have shown that $D_0 = 0$ and $c_i = 0~(2 \leq i \leq n)$. Therefore, $c_1D_1 = 0$. Thanks again to Lemma \ref{lm:bl-prop} and the standard linear transformation to the $(n-1)$-dimensional reference element $\hat{T}_{n-1}$, we have 
$$ 
\begin{aligned}
D_1 &= (n-1)! \int_{\hat{T}_{n-1}} \int_0^{\hat{x}_1} b_{\ell_0}(t)\,\mathrm{d}t \mathrm{d} \hat{x}_1 \cdots \mathrm{d}\hat{x}_{n-1} \\
&= (n-1) \int_0^1 (1-\hat{x}_1)^{n-2}\int_0^{\hat{x}_1} b_{\ell_0}(t)\,\mathrm{d}t \mathrm{d} \hat{x}_1 \\
&= \int_0^1 (1-t)^{n-1} b_{\ell_0}(t) \,\mathrm{d}t   \neq 0,
\end{aligned}
$$ 
whence $c_1 = 0$.

Combining {\it Step 2-1} to {\it Step 2-4}, we obtain that $v=0$ for {\it Case 2}. This completes the proof.
  \end{proof}

\section{Application to $2m$-th order elliptic PDEs} \label{sec:2m}

As a direct application of the proposed finite elements, this section presents the nonconforming finite element method and theoretical analysis for solving the $m$-harmonic problem \eqref{eq:m-harmonic}. The primary analytical approach employed here is well-established (c.f. \cite{wang2013minimal, wu2019nonconforming}), so we will not repeat the detailed proofs. Instead, we provide an overview of the main steps and emphasize the key points in the process.

\subsection{Finite element space}
We define the spaces $V_h^{(m, n)}$ and $V_{h0}^{(m,n)}$ as follows: 
\begin{enumerate}
\item $V_h^{(m,n)}$ consists of all functions $v_h|_T \in P_{T}^{(m,n)}$, such that any $0 \leq \ell \leq \lceil \frac{m}{n} \rceil -1$, any
$1 \leq k \leq \min\{m - \ell n, n\}$, any $(n-k)$-dimensional sub-simplex $F$ of any $T\in \mathcal{T}_h$ and any
$\balpha \in A_k$ with $|\balpha|=m - \ell n-k$, $d_{T,F,\balpha}(v_h)$ is
continuous through $F$. 

\item $V_{h0}^{(m,n)} \subset V_h^{(m,n)}$ such that for any $v_h
\in V_{h0}^{(m,n)}$, $d_{T,F,\balpha}(v_h) = 0$ if the $(n-k)$-dimensional sub-simplex
$F\subset \partial \Omega$.
\end{enumerate}

{\it Approximation property}. Based on Theorem \ref{tm:unisolvence} (unisolvence), we define the interpolation operator $\Pi_T: H^{m}(T) \to P_T^{(m,n)}$ by 
$$
\Pi_T v = \sum_{i=1}^J p_i d_{T,i}(v) \quad \forall v\in H^{m}(T),
$$
where $p_i \in P_T^{(m,n)}$ is the nodal basis function that satisfies $d_{T,j}(p_i) = \delta_{ij}$, and $\delta_{ij}$ is the Kronecker delta. 
The standard interpolation theory (cf. \cite{ciarlet1978finite, brenner2007mathematical}) implies that, for $s \in [0,1]$ and any integer $0\leq k \leq m$,  
\begin{equation} \label{eq:interpolation-local}
|v - \Pi_T v|_{k,T} \lesssim h_T^{m+s-k}|v|_{m+s,T} \quad \forall v\in H^{m+s}(T).
\end{equation}

The global interpolation operator $\Pi_h$ on $H^{m}(\Omega)$ is
defined as $(\Pi_h v)|_T := \Pi_T(v|_T)$,  for all $T\in \mathcal{T}_h, v \in H^{m}(\Omega)$.
By the above definition, we immediately have $\Pi_hv \in V_h^{(m,n)}$ for any $v\in H^{m}(\Omega)$ and $\Pi_h v \in V_{h0}^{(m,n)}$ for any $v\in H_0^{m}(\Omega)$. The approximate property of $V_h^{(m,n)}$ and $V_{h0}^{(m,n)}$ then follows directly from \eqref{eq:interpolation-local}, whose proof follows the same argument in
\cite[Theorem 2.1]{wang2013minimal}.

\begin{theorem}[approximability] \label{tm:interpolation-global}
For $s \in [0,1]$, it holds that 
\begin{equation} \label{eq:interpolation-global}
\|v - \Pi_h v\|_{m,h} \lesssim h^{s}|v|_{m+s,\Omega} \quad \forall
v\in H^{m+s}(\Omega),
\end{equation}
and for any $v \in H^{m}(\Omega)$, $\lim_{h\to 0}\|v - \Pi_h v\|_{m,h} = 0$.
\end{theorem}

{\it Weak continuity}. Let $1 \leq k \leq n$ and $F$ be an
$(n-k)$-dimensional sub-simplex of $T \in \mathcal{T}_h$. Thanks to Lemma \ref{lm:LW-derivative} (equivalence for vanishing DOF), for any $v_h \in V_h^{(m,n)}$ and any $T' \in \mathcal{T}_h$ with $F \subset T'$, we immediately have 
$$
\int_F \partial^{\balpha}(v_h|_{T}) = \int_F \partial^{\balpha}(v_h|_{T'}) \qquad \forall |\balpha| = m - \ell n - k, 1 \leq \ell \leq \lceil \frac{m}{n} \rceil -1.
$$ 
Further, if $F \subset \partial \Omega$, then for any $v_h \in V_{h0}^{(m,n)}$, it holds that $\int_F \partial^{\balpha}(v_h|_T) = 0$, for all $|\balpha| = m - \ell n - k, 1 \leq \ell \leq \lceil \frac{m}{n} \rceil -1$. Therefore, the weak continuity and weak zero-boundary conditions (see Definition \ref{df:weak-continuity}) hold for the proposed finite element spaces. Using the method outlined in \cite{wang2001necessity}, this condition can lead to the following lemma. The key technique involves averaging the degrees of freedom in the $H^1$-Lagrange finite element space, as detailed in the proof of \cite[Lemma 3.1]{wang2013minimal}.

\begin{lemma}[$H^1$ weak approximation] \label{lm:H1-weak-approximation}
For any $v_h \in V_h^{(m,n)}$ and $|\balpha| < m$, there exists a piecewise polynomial $v_{\balpha} \in H^1(\Omega)$
such that  
$$
|\partial_h^\balpha v_h - v_\balpha|_{j,h} \lesssim h^{m-|\balpha|-j}
|v_h|_{m,h} \qquad 0 \leq j \leq m-|\balpha|,
$$
Further, $v_{\balpha}$ can be chosen in $H_0^1(\Omega)$ when $v_h \in V_{h0}^{(m,n)}$.
\end{lemma}

This property has a twofold application: first, by utilizing the stability of $v_{\balpha}$ and the Poincar\'{e} inequality for $H^1$ functions, we obtain
\begin{equation} \label{eq:Poincare}
\begin{aligned}
\|v_h\|_{m,h} &\lesssim |v_h|_{m,h} \quad \forall v_h \in
V_{h0}^{(m,n)}, \\
\|v_h\|_{m,h}^2 &\lesssim |v_h|_{m,h}^2 + \sum_{|\balpha|<m}
\left( \int_\Omega \partial_h^\balpha v_h\right)^2 \quad \forall v_h
\in V_h^{(m,n)}.
\end{aligned}
\end{equation}
Secondly, using the approximation property of $v_{\balpha}$ with respect to derivatives less than $m$, the lower-order terms in the consistency error estimate can be directly obtained (see Theorem \ref{tm:error-regularity} and Theorem \ref{tm:error-no-regularity} below).

\subsection{Nonconforming finite element method}
We denote $V_h := V_{h0}^{(m,n)}$ as the nonconforming approximation of
$H_0^{m}(\Omega)$. Then, the nonconforming finite element method for problem \eqref{eq:m-harmonic} is to find
$u_h \in V_h$, such that 
\begin{equation} \label{eq:nonconforming-FEM}
a_h(u_h, v_h) = (f, v_h) \quad \forall v_h \in V_h.
\end{equation}
Here, the broken bilinear form $a_h(\cdot, \cdot)$ is defined as 
$$ 
\begin{aligned}
a_h(v, w) &:= (\nabla_h^m v, \nabla_h^m w) \\
&= \sum_{T\in \mathcal{T}_h}
\int_T \sum_{|\balpha|=m} \frac{m!}{\alpha_1! \cdots \alpha_n!}  \partial^\balpha v \partial^\balpha w 
\quad \forall v,w\in H^m(\Omega) + V_h.
\end{aligned}
$$ 

From the Poincar\'{e} inequality \eqref{eq:Poincare}, the bilinear form $a_h(\cdot, \cdot)$ is uniformly $V_h$-elliptic. For the consistent condition, we apply the generalized patch test proposed in \cite{stummel1979generalized} to obtain the following theorem. Other sufficient conditions that are easier to achieve can also be used, such as the patch test \cite{bazeley1965triangular, irons1972experience, veubeke1974variational, wang2001necessity}, the weak patch test \cite{wang2001necessity}, and the F-E-M test \cite{shi1987fem, hu2014new}. 

\begin{theorem}[convergence for $L^2$ data] \label{tm:convergence}
For any $f\in L^2(\Omega)$, the solution $u_h$ of problem
\eqref{eq:nonconforming-FEM} converges to the solution of
\eqref{eq:m-harmonic} for any $m \geq 1$, i.e., 
$$ 
\lim_{h \to 0} \|u-u_h\|_{m,h} = 0.
$$ 
\end{theorem}

\subsection{Error estimate}
Based on Strang's Lemma, the error $|u - u_h|_{m,h}$ can be controlled by the approximation error and the consistency error.
Following \cite{wu2019nonconforming}, we present two types of the estimate for the consistent error.

{\it Error estimate under the extra regularity assumption}.  Let $r = \max\{m+1, 2m-1\}$. If $u \in H^r(\Omega)$ and $f \in L^2(\Omega)$, then the consistency error estimate is given by
$$
\sup_{v_h \in V_h} \frac{|a_h(u, v_h) - \langle f, v_h \rangle |}{|v_h|_{m,h}} \lesssim \sum_{k=1}^{r-m} h^k |u|_{m+k} + h^m \|f\|_0.
$$
The proof can be found in \cite[Lemma 3.2]{wang2013minimal} or \cite[Lemma 3.2]{wu2019nonconforming}. Combined with Theorem \ref{tm:interpolation-global} (approximability), we obtain the following estimate.
\begin{theorem}[error estimate I] \label{tm:error-regularity}
Let $r = \max\{m+1, 2m-1\}$. If $u \in H^r(\Omega)$ and $f \in
L^2(\Omega)$, then 
\begin{equation} \label{equ:error-regularity}
|u - u_h|_{m,h} \lesssim \sum_{k=1}^{r-m} h^k |u|_{m+k} + h^m \|f\|_0. 
\end{equation}
\end{theorem}
Note that when $n=1$, all cases reduce to conforming elements, so $|u - u_h|_{m,h}$ is determined by the approximation error, then Theorem \ref{tm:error-regularity} can be improved as:
$$
|u - u_h|_{m,h} \lesssim h^{r - m} |u|_r \quad \forall u \in H^r(\Omega), r = \max\{m+1, 2m-1\}.
$$ 
However, when $n \geq 2$, the finite element spaces are not even $C^0$ continuous, indicating that the result in Theorem \ref{tm:error-regularity} cannot be further improved.

{\it Error estimate by conforming relatives}. The error estimate can be improved with minimal regularity under the following assumption.
\begin{assumption}[conforming relatives] \label{asm:conforming-relatives}
There exists an $H^m$-conforming finite element space $V_h^c \subset
H_0^m(\Omega)$, and an operator $\Pi_h^c: V_h \mapsto V_h^c$ such that 
\begin{equation} \label{eq:conforming-relatives}
\sum_{j=0}^{m-1} h^{2(j-m)} |v_h - \Pi_h^c v_h|_{j,h}^2 + |\Pi_h^c
v_h|_{m,h}^2 \lesssim |v_h|_{m,h}^2. 
\end{equation} 
\end{assumption}

The essence of this assumption lies in the interpolation on $H^m$-conforming elements using degrees of freedom averaging techniques. In fact, it has been verified in various cases: see \cite{scott1990finite, brenner2003poincare} for the case when $m = 1$, \cite{li2014new} for the Morley element in 2D, and \cite{hu2014new, hu2016canonical} for arbitrary $m \geq 1$ in 2D. For higher-dimensional cases, the validation of this assumption is implicit in the construction of arbitrary $H^m$ conforming elements in any dimension \cite{hu2023construction}.

Under Assumption \ref{asm:conforming-relatives}, if $f \in L^2(\Omega)$, we follow the proof of \cite[Lemma 3.7]{wu2019nonconforming} to obtain 
$$
\begin{aligned}
\sup_{v_h \in V_h} & \frac{|a_h(u, v_h) - \langle f, v_h \rangle
|}{|v_h|_{m,h}} \lesssim \inf_{w_h\in V_h} |u-w_h|_{m,h} +
 h^m \|f\|_0 \\
&+ \sum_{|\balpha|=m} \left(\|\partial^\balpha u - P_h^0
    \partial^\balpha u\|_0 + \sum_{F\in \mathcal{F}_h}
    \|\partial^\balpha u - P_{\omega_F}^0 \partial^\balpha
    u\|_{0,\omega_F} \right).
\end{aligned}
$$
Here, $P_h^0$ denotes the $L^2$ projection onto the piecewise constant space, while $P_{\omega_F}^0$ represents the local $L^2$ projection onto the constant space, where $\omega_F$ is the union of all elements that share the face $F$. Combining with approximation properties, we obtain the following estimate.

\begin{theorem}[error estimate II] \label{tm:error-no-regularity}
Under Assumption \ref{asm:conforming-relatives}, if $f \in
L^2(\Omega)$ and $u \in H^{m+t}(\Omega)$, then 
\begin{equation}\label{equ:error-no-regularity}
|u - u_h|_{m,h} \lesssim h^{s}|u|_{m+s} + h^m \|f\|_0,
\end{equation}
where $s = \min\{1,t\}$.
\end{theorem}

\section{Numerical Experiments} \label{sec:numerical}
In this section, we present several 2D numerical results to support the
theoretical results obtained in Section \ref{sec:2m}.

\subsection{Example 1: smooth solution}
In the first numerical example, we choose $f$ such that the exact solution is $u
= 2^{4m-6} (x-x^2)^m (y-y^2)^m$ in $\Omega = (0,1)^2$, which provides the
homogeneous boundary conditions in \eqref{eq:m-harmonic}. After computing 
\eqref{eq:nonconforming-FEM} for various values of $h$, we calculate
the errors and orders of convergence in $H^{k} (k=0,1,2,\ldots,m)$ and report
the results of $m=3,4$ in Table \ref{tab:example1_m3} and Table \ref{tab:example1_m4}, respectively. The tables show that the computed
solution converges linearly to the exact solution in the broken $H^m$ norm, which
is in agreement with Theorem \ref{tm:error-regularity} and
Theorem \ref{tm:error-no-regularity}.  Further, Table
\ref{tab:example1_m3} and Table \ref{tab:example1_m4} indicate that the convergence order of the error in lower-order norms is $h^2$.

\begin{table}[!htbp]
\caption{Example 1 $(m=3)$: Errors and convergence orders.}
\centering
{\small{
\begin{tabular}{@{}c|cc|cc|cc|cc@{}}
   \hline
  $1/h$	&$\|u-u_h\|_0$	& Order	& 
  $|u-u_h|_{1,h}$ & Order & 
  $|u-u_h|_{2,h}$	& Order & $|u-u_h|_{3,h}$ & Order\\ \hline
  4		&2.1506e-3 &--	  &1.5144e-2 &--	  &1.3957e-1 &--   &2.4820e+0	&--  \\
  8		&1.9903e-3 &0.11	&1.0276e-2 &0.56	&6.2813e-2 &1.15 &1.4448e+0	&0.78\\
  16	&6.3643e-4 &1.64	&3.1633e-3 &1.70	&1.9066e-2 &1.72 &7.6583e -1	&0.91\\
  32	&1.6858e-4 &1.92	&8.3252e-4 &1.93	&5.0312e-3 &1.92 &3.8912e -1	&0.98\\
  64	&4.2755e-5 &1.98	&2.1091e-4 &1.98	&1.2762e-3 &1.98 &1.9536e -1	&0.99\\
  \hline
\end{tabular}
}}
\label{tab:example1_m3}
\end{table}

\begin{table}[!htbp]
    \caption{Example 1 $(m=4)$: Errors and convergence orders.}
    \centering
    {\small{

 \begin{tabular}{@{}c|cc|cc|cc@{}}
    \hline
   $1/h$	&$\|u-u_h\|_0$	& Order	& 
   $|u-u_h|_{1,h}$ & Order &  $|u-u_h|_{2,h}$	& Order \\ \hline
   4		&2.6832e-3 &--	  &1.6055e-2 &--	  &1.6847e-1 &-- \\
   8	&1.7536e-3 &0.61	&1.1231e-2 &0.52	&9.8257e-2 &0.78 \\
   16	&8.5302e-4 &1.04	&4.8519e-3 &1.21	&3.8117e-2 &1.37 \\
   32	&2.4791e-4 &1.78	&1.3830e-3 &1.81	&1.0665e-2 &1.84\\
   64	&5.4171e-5 &2.19	&2.9635e-4 &2.22	&2.2511e-3 &2.24 \\
   \hline
 \end{tabular}

 \begin{tabular}{c|cc|cc@{}}
    \hline
   $1/h$	&
   $\quad|u-u_h|_{3,h}\quad$ & $\quad$ Order $\quad$ &  $\quad|u-u_h|_{4,h}\quad$	& $\quad$ Order $\ \ $ \\ \hline
   4		   &2.2146e+0	&--  &3.9478e+1	&-- \\
   8	 &8.9968e -1	&1.30 &2.4686e+1	&0.68\\
   16	 &3.3377e -1	&1.43 &1.3437e+1	&0.88\\
   32	 &9.4056e -2	&1.83 &6.9258e+0	&0.96\\
   64	&2.1248e -2 &2.15	&3.4834e+0 &0.99\\
   \hline
 \end{tabular}
 }}
 \label{tab:example1_m4}
 \end{table}

\begin{figure}[!htbp]
\caption{Meshes for Example 1 (left) and Example 2 (right).}
\label{fig:uniform-grids}
\centering 
\captionsetup{justification=centering}
   \includegraphics[width=0.35\textwidth]{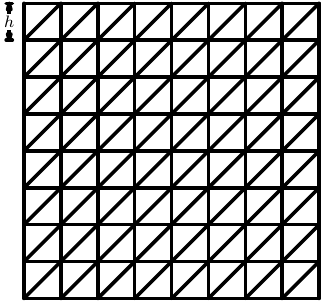} 
\qquad\qquad  
   \includegraphics[width=0.35\textwidth]{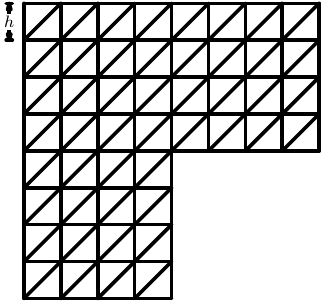} 
\end{figure}

\subsection{Example 2}
In the second example, we test the method in which the solution has
partial regularity on a non-convex domain. To this end, we solve the
$m$-harmonic equation $(-\Delta)^m u = 0$ 
on the 2D L-shaped domain $\Omega = (-1,1)^2\setminus [0,1)\times
(-1,0]$ shown in Figure \ref{fig:uniform-grids}, with Dirichlet boundary conditions
given by the exact solution 
$$ 
u = r^{m-1/2} \sin((m-1/2)\theta),
$$ 
where $(r,\theta)$ are polar coordinates. Due to the singularity at
the origin, the solution $u \in H^{m+1/2-\varepsilon}(\Omega)$, for any $\varepsilon>0$.
The method does converge with the optimal order $h^{1/2-\varepsilon}$ in the broken $H^m$
norm, as shown in Table \ref{tab:example2_m3} and  \ref{tab:example2_m4},
 where the results of $m=3,4$ is presented.

\begin{table}[!htbp]
    \caption{Example 2 $(m=3)$: Errors and convergence orders.}
    \centering
    {\small{
    \begin{tabular}{@{}c|cc|cc|cc|cc@{}}
       \hline
      $1/h$	&$\|u-u_h\|_0$	& Order	& 
      $|u-u_h|_{1,h}$ & Order & 
      $|u-u_h|_{2,h}$	& Order & $|u-u_h|_{3,h}$ & Order\\ \hline
      4		&9.7591e-4 &--	  &1.0280e-2 &--	  &1.0522e-1 &--   &2.1435e+0	&--  \\
      8		&4.4795e-4 &1.12	&2.8791e-3 &1.84	&4.1464e-2 &1.34 &1.4583e+0	&0.56\\
      16	&2.0399e-4 &1.13	&1.1253e-3 &1.36	&1.6098e-2 &1.36 &1.0330e+0	&0.50\\
      32	&8.9502e-5 &1.19	&4.9272e-4 &1.19	&6.3490e-2 &1.34 &7.3245e-1	&0.50\\
      64	&4.0176e-5 &1.16	&2.2207e-4 &1.15	&2.5788e-3 &1.30 &5.1862e-1	&0.50\\
      \hline
    \end{tabular}
    }}
    \label{tab:example2_m3}
    \end{table}

    \begin{table}[!htbp]
        \caption{Example 2 $(m=4)$: Errors and convergence orders.}
        \centering
        {\small{

     \begin{tabular}{@{}c|cc|cc|cc@{}}
        \hline
       $1/h$	&$\|u-u_h\|_0$	& Order	& 
       $|u-u_h|_{1,h}$ & Order &  $|u-u_h|_{2,h}$	& Order \\ \hline
       4		&5.8990e-4 &--	  &4.4638e-3 &--	  &5.2416e-2 &--\\
       8	&1.5666e-4 &1.91	&9.0894e-4 &2.30	&1.0984e-2 &2.25\\
       16	&6.4249e-5 &1.29	&3.7094e-4 &1.29	&3.2422e-3 &1.76\\
       32	&2.4888e-5 &1.37	&1.4626e-5 &1.34	&1.1898e-3 &1.45\\
       \hline
     \end{tabular}
    
     \centering
     \begin{tabular}{c|cc|cc@{}}
        \hline
       $1/h$	&
       $\quad|u-u_h|_{3,h}\quad$ & $\quad$ Order $\quad$ &  $\quad|u-u_h|_{4,h}\quad$	& $\quad$ Order $\ \ $ \\ \hline
       4		&5.5950e-1	&--  &1.1965e+1	&-- \\
       8	&2.0456e-1	&1.45 &7.4219e+0	&0.69\\
       16	&7.6405e-2	&1.42 &5.1475e+0	&0.53\\
       32	&2.8363e-2	&1.43 &3.6424e+0	&0.50\\
       \hline
     \end{tabular}
     }}
     \label{tab:example2_m4}
     \end{table}

\section{Concluding remarks} \label{sec:remarks}
In this paper, a family of lowest-order nonconforming finite element spaces is designed for solving $2m$-th order elliptic problems on $n$-dimensional simplicial grids for any $m, n \geq 1$. This paper provides a solution to the open problem of overcoming the restriction $m \leq n$ in the well-known Morley-Wang-Xu element \cite{wang2013minimal}. To the best of the authors' knowledge, this is also the first result in the canonical piecewise polynomial setting of nonconforming element for any $m,n \geq 1$.

Considering the weak continuity in the construction of nonconforming elements, it is our believe that the design of the degrees of freedom in \eqref{eq:DOF} represents the ``optimal'' choice, as they are maximally utilized through integration by parts. On the other hand, the local number of the degrees of freedom in \eqref{eq:number-dof} is entirely comparable to the minimal requirements of polynomial approximation $\mathcal{P}_m(T)$. In fact, a rough estimate of \eqref{eq:number-dof} shows that for any $m > n \geq 1$, $\dim P_T^{(m,n)} \leq \dim \mathcal{P}_m(T) + \dim \mathcal{P}_{m-n}(T) \leq 2\dim \mathcal{P}_m(T)$, indicating that the number of local degrees of freedom in the proposed elements is optimal, up to a uniform constant factor. 



\bibliographystyle{amsplain}
\bibliography{LW}   

\providecommand{\bysame}{\leavevmode\hbox to3em{\hrulefill}\thinspace}
\providecommand{\MR}{\relax\ifhmode\unskip\space\fi MR }
\providecommand{\MRhref}[2]{%
  \href{http://www.ams.org/mathscinet-getitem?mr=#1}{#2}
}
\providecommand{\href}[2]{#2}
\begin{thebibliography}{10}

\bibitem{adams2003sobolev}
Robert~A Adams and John~JF Fournier, \emph{Sobolev spaces}, vol. 140, Academic
  press, 2003.

\bibitem{antonietti2020conforming}
Paola~F Antonietti, Gianmarco Manzini, and Marco Verani, \emph{The conforming
  virtual element method for polyharmonic problems}, Computers \& Mathematics
  with Applications \textbf{79} (2020), no.~7, 2021--2034.

\bibitem{arnold1985mixed}
Douglas~N Arnold and Franco Brezzi, \emph{Mixed and nonconforming finite
  element methods: implementation, postprocessing and error estimates},
  RAIRO-Mod{\'e}lisation math{\'e}matique et analyse num{\'e}rique \textbf{19}
  (1985), no.~1, 7--32.

\bibitem{bazeley1965triangular}
G.~P. Bazeley, Y.-K. Cheung, B.~M. Irons, and O.~C. Zienkiewicz,
  \emph{Triangular elements in plate bending conforming and nonconforming
  solutions}, Proceedings of the Conference on Matrix Methods in Structural
  Mechanics, Wright Patterson AF Base, Ohio, 1965, pp.~547--576.

\bibitem{beirao2013basic}
L~Beir{\~a}o~da Veiga, Franco Brezzi, Andrea Cangiani, Gianmarco Manzini,
  L~Donatella Marini, and Alessandro Russo, \emph{Basic principles of virtual
  element methods}, Mathematical Models and Methods in Applied Sciences
  \textbf{23} (2013), no.~01, 199--214.

\bibitem{beirao2014hitchhiker}
L~Beir{\~a}o~da Veiga, Franco Brezzi, Luisa~Donatella Marini, and Alessandro
  Russo, \emph{The hitchhiker's guide to the virtual element method},
  Mathematical models and methods in applied sciences \textbf{24} (2014),
  no.~08, 1541--1573.

\bibitem{bramble1970triangular}
James~H Bramble and Milo{\v{s}} Zl{\'a}mal, \emph{Triangular elements in the
  finite element method}, Mathematics of Computation \textbf{24} (1970),
  no.~112, 809--820.

\bibitem{brenner2003poincare}
Susanne~C Brenner, \emph{{Poincar{\'e}--Friedrichs inequalities for piecewise
  $H^1$ functions}}, SIAM Journal on Numerical Analysis \textbf{41} (2003),
  no.~1, 306--324.

\bibitem{brenner2007mathematical}
Susanne~C Brenner and L~Ridgway Scott, \emph{The mathematical theory of finite
  element methods}, vol.~15, Springer Science \& Business Media, 2007.

\bibitem{brenner2005c0}
Susanne~C Brenner and Li-Yeng Sung, \emph{{$C^0$ interior penalty methods for
  fourth order elliptic boundary value problems on polygonal domains}}, Journal
  of Scientific Computing \textbf{22} (2005), no.~1, 83--118.

\bibitem{chen2022conforming}
Chunyu Chen, Xuehai Huang, and Huayi Wei, \emph{Conforming virtual elements in
  arbitrary dimension}, SIAM Journal on Numerical Analysis \textbf{60} (2022),
  no.~6, 3099--3123.

\bibitem{chen2022c0}
Huangxin Chen, Jingzhi Li, and Weifeng Qiu, \emph{{A $C^0$ interior penalty
  method for $m$th-Laplace equation}}, ESAIM: Mathematical Modelling and
  Numerical Analysis \textbf{56} (2022), no.~6, 2081--2103.

\bibitem{chen2020nonconforming}
Long Chen and Xuehai Huang, \emph{Nonconforming virtual element method for
  $2m$th order partial differential equations in $\mathbb{R}^n$}, Mathematics
  of Computation \textbf{89} (2020), no.~324, 1711--1744.

\bibitem{chen2021geometric}
\bysame, \emph{Geometric decompositions of the simplicial lattice and smooth
  finite elements in arbitrary dimension}, arXiv preprint arXiv:2111.10712
  (2021).

\bibitem{ciarlet1978finite}
Philippe~G Ciarlet, \emph{The finite element method for elliptic problems},
  North-Holland, 1978.

\bibitem{da2020c1}
L~Beirao da~Veiga, Franco Dassi, and Alessandro Russo, \emph{{A $C^1$ virtual
  element method on polyhedral meshes}}, Computers \& Mathematics with
  Applications \textbf{79} (2020), no.~7, 1936--1955.

\bibitem{veubeke1974variational}
B~Fraeijs De~Veubeke, \emph{Variational principles and the patch test},
  International Journal for Numerical Methods in Engineering \textbf{8} (1974),
  no.~4, 783--801.

\bibitem{droniou2017mixed}
J{\'e}r{\^o}me Droniou, Muhammad Ilyas, Bishnu~P Lamichhane, and Glen~E
  Wheeler, \emph{A mixed finite element method for a sixth-order elliptic
  problem}, IMA Journal of Numerical Analysis \textbf{39} (2019), no.~1,
  374--397.

\bibitem{gudi2011interior}
Thirupathi Gudi and Michael Neilan, \emph{{An interior penalty method for a
  sixth-order elliptic equation}}, IMA Journal of Numerical Analysis
  \textbf{31} (2011), no.~4, 1734--1753.

\bibitem{hu2023construction}
Jun Hu, Ting Lin, and Qingyu Wu, \emph{{A construction of $C^r$ conforming
  finite element spaces in any dimension}}, Foundations of Computational
  Mathematics (2023), 1--37.

\bibitem{hu2024condition}
Jun Hu, Ting Lin, Qingyu Wu, and Beihui Yuan, \emph{The condition for
  constructing a finite element from a superspline}, arXiv preprint
  arXiv:2407.03680 (2024).

\bibitem{hu2014new}
Jun Hu, Rui Ma, and Zhong-Ci Shi, \emph{A new a priori error estimate of
  nonconforming finite element methods}, Science China Mathematics \textbf{57}
  (2014), no.~5, 887--902.

\bibitem{hu2020family}
Jun Hu, Shudan Tian, and Shangyou Zhang, \emph{{A family of 3D
  $H^2$-nonconforming tetrahedral finite elements for the biharmonic
  equation}}, Science China Mathematics \textbf{63} (2020), 1505--1522.

\bibitem{hu2015minimal}
Jun Hu and Shangyou Zhang, \emph{{The minimal conforming $H^k$ finite element
  spaces on $\mathcal{R}^n$ rectangular grids}}, Mathematics of Computation
  \textbf{84} (2015), no.~292, 563--579.

\bibitem{hu2016canonical}
\bysame, \emph{{A canonical construction of $H^m$-nonconforming triangular
  finite elements}}, Annals of Applied Mathematics \textbf{33} (2017), no.~3,
  266--288.

\bibitem{huang2020nonconforming}
Xuehai Huang, \emph{Nonconforming virtual element method for $2m$th order
  partial differential equations in $\mathbb{R}^n$ with $m>n$}, Calcolo
  \textbf{57} (2020), no.~4, 42.

\bibitem{irons1972experience}
Bruce~M Irons and Abdur Razzaque, \emph{Experience with the patch test for
  convergence of finite elements}, The mathematical foundations of the finite
  element method with applications to partial differential equations (1972),
  557--587.

\bibitem{jin2023two}
Xianlin Jin and Shuonan Wu, \emph{Two families of $n$-rectangle nonconforming
  finite elements for sixth-order elliptic equations}, Journal of Computational
  Mathematics (2024).

\bibitem{li1999full}
Jichun Li, \emph{Full-order convergence of a mixed finite element method for
  fourth-order elliptic equations}, Journal of Mathematical Analysis and
  Applications \textbf{230} (1999), no.~2, 329--349.

\bibitem{li2006optimal}
\bysame, \emph{Optimal convergence analysis of mixed finite element methods for
  fourth-order elliptic and parabolic problems}, Numerical Methods for Partial
  Differential Equations \textbf{22} (2006), no.~4, 884--896.

\bibitem{li2014new}
Mingxia Li, Xiaofei Guan, and Shipeng Mao, \emph{{New error estimates of the
  Morley element for the plate bending problems}}, Journal of Computational and
  Applied Mathematics \textbf{263} (2014), 405--416.

\bibitem{schedensack2016new}
Mira Schedensack, \emph{{A new discretization for $m$th-Laplace equations with
  arbitrary polynomial degrees}}, SIAM Journal on Numerical Analysis
  \textbf{54} (2016), no.~4, 2138--2162.

\bibitem{scott1990finite}
L~Ridgway Scott and Shangyou Zhang, \emph{Finite element interpolation of
  nonsmooth functions satisfying boundary conditions}, Mathematics of
  computation \textbf{54} (1990), no.~190, 483--493.

\bibitem{shi1987fem}
Zhong-Ci Shi, \emph{{The F-E-M test for convergence of nonconforming finite
  elements}}, Mathematics of Computation \textbf{49} (1987), no.~180, 391--405.

\bibitem{stummel1979generalized}
Friedrich Stummel, \emph{The generalized patch test}, SIAM Journal on Numerical
  Analysis \textbf{16} (1979), no.~3, 449--471.

\bibitem{walkington2014c}
Noel~J Walkington, \emph{A {$C^1$} tetrahedral finite element without edge
  degrees of freedom}, SIAM Journal on Numerical Analysis \textbf{52} (2014),
  no.~1, 330--342.

\bibitem{wang2001necessity}
Ming Wang, \emph{On the necessity and sufficiency of the patch test for
  convergence of nonconforming finite elements}, SIAM Journal on Numerical
  Analysis \textbf{39} (2001), no.~2, 363--384.

\bibitem{wang2007new}
Ming Wang, Zhong-Ci Shi, and Jinchao Xu, \emph{A new class of
  {Z}ienkiewicz-type nonconforming element in any dimensions}, Numerische
  Mathematik \textbf{106} (2007), no.~2, 335--347.

\bibitem{wang2007some}
\bysame, \emph{Some n-rectangle nonconforming elements for fourth order
  elliptic equations}, Journal of Computational Mathematics (2007), 408--420.

\bibitem{wang2006morley}
Ming Wang and Jinchao Xu, \emph{The {M}orley element for fourth order elliptic
  equations in any dimensions}, Numerische Mathematik \textbf{103} (2006),
  no.~1, 155--169.

\bibitem{wang2013minimal}
\bysame, \emph{{Minimal finite element spaces for $2m$-th-order partial
  differential equations in $\mathbb{R}^n$}}, Mathematics of Computation
  \textbf{82} (2013), no.~281, 25--43.

\bibitem{wu2017pm}
Shuonan Wu and Jinchao Xu, \emph{$\mathcal{P}_m$ interior penalty nonconforming
  finite element methods for $2m$-th order {PDEs} in $\mathbb{R}^n$}, arXiv
  preprint arXiv:1710.07678 (2017).

\bibitem{wu2019nonconforming}
\bysame, \emph{Nonconforming finite element spaces for $2m$th order partial
  differential equations on $\mathbb{R}^n$ simplicial grids when $m=n+1$},
  Mathematics of Computation \textbf{88} (2019), no.~316, 531--551.

\bibitem{vzenivsek1970interpolation}
Alexander {\v{Z}}en{\'\i}{\v{s}}ek, \emph{Interpolation polynomials on the
  triangle}, Numerische Mathematik \textbf{15} (1970), no.~4, 283--296.

\bibitem{zhang2009family}
Shangyou Zhang, \emph{{A family of 3D continuously differentiable finite
  elements on tetrahedral grids}}, Applied Numerical Mathematics \textbf{59}
  (2009), no.~1, 219--233.

\bibitem{zhang2016family}
\bysame, \emph{A family of differentiable finite elements on simplicial grids
  in four space dimensions}, Mathematica Numerica Sinica \textbf{38} (2016),
  no.~3, 309--324.

\bibitem{zhang2022nodal}
\bysame, \emph{The nodal basis of {$C^m$-$P_k^{(3)}$ and $C^m$-$P_k^{(4)}$}
  finite elements on tetrahedral and {4D} simplicial grids}, arXiv preprint
  arXiv:2202.05837 (2022).

\end{thebibliography}

\end{document}